\documentclass[12pt,reqno]{amsart}
\usepackage[a4paper,left=3.25cm,right=3.25cm]{geometry}
\title{Categorification and the quantum Grassmannian}
\date{23 May 2022}

\author{Bernt Tore Jensen}
\address{BTJ: Department of Mathematical Sciences, Norwegian University of Science and Technology 
Gj\o vik, Teknologivegen 22, 2815 Gj\o vik, Norway}
\email{bernt.jensen@ntnu.no}

\author{Alastair King}
\address{AK: Mathematical Sciences, Univ. of Bath, Bath BA2 7AY, U.K.}
\email{a.d.king@bath.ac.uk}

\author{Xiuping Su}
\address{XS: Mathematical Sciences, Univ. of Bath, Bath BA2 7AY, U.K.}
\email{xs214@bath.ac.uk}

\usepackage[dvipsnames]{xcolor}
\usepackage[colorlinks]{hyperref}
\usepackage{url}

\usepackage{tikz}
  \usetikzlibrary{calc,decorations.markings,arrows.meta,arrows}
  \tikzset{equals/.style={double=none, double distance=2pt}, 
     cdmap/.style={black, thick, -angle 60}} 
\newcommand{\ph}{\phantom{i}}
\newcommand{\cdzero}{0_{\ph}}
\newcommand{\maplab}{\small}
 
\usepackage{amssymb,amsthm,array}

\theoremstyle{plain}
\newtheorem{theorem}{Theorem}[section]
\newtheorem{proposition}[theorem]{Proposition}
\newtheorem{lemma}[theorem]{Lemma}

\theoremstyle{definition}
\newtheorem{definition}[theorem]{Definition}

\newtheorem{remark}[theorem]{Remark}

\numberwithin{equation}{section}

\renewcommand{\setminus}{\smallsetminus}
\renewcommand{\epsilon}{\varepsilon}

\renewcommand{\leq}{\leqslant}
\renewcommand{\geq}{\geqslant}

\newcommand{\sfrac}[2]{{\textstyle\frac{#1}{#2}}}

\newcommand{\longto}{\longrightarrow}
\newcommand{\isom}{\cong}
\newcommand{\sub}{\subseteq}

\newcommand{\isoto}{\stackrel{\isom}{\longrightarrow}}
\newcommand{\isofrom}{\stackrel{\isom}{\longleftarrow}}

\newcommand{\CC}{\mathbb{C}}
\newcommand{\NN}{\mathbb{N}}

\newcommand{\ZZ}{\mathbb{Z}}

\newcommand{\cF}{\mathcal{F}}
\newcommand{\cT}{\mathcal{T}}
\renewcommand{\a}{\textbf{a}}
\renewcommand{\b}{\textbf{b}}

\newcommand{\CM}{\operatorname{CM}}
\newcommand{\md}{\operatorname{mod}}
\newcommand{\Sub}{\operatorname{Sub}}
\newcommand{\Fac}{\operatorname{Fac}}

\newcommand{\Hom}{\operatorname{Hom}}

\newcommand{\End}{\operatorname{End}}
\newcommand{\Ext}{\operatorname{Ext}}
\newcommand{\add}{\operatorname{add}}

\newcommand{\coker}{\operatorname{coker}}
\newcommand{\rk}{\operatorname{rk}}
\newcommand{\len}{\operatorname{len}}

\renewcommand{\dim}{\operatorname{dim}}
\newcommand{\Irr}{\operatorname{Irr}}
\newcommand{\rad}{\operatorname{rad}}

\newcommand{\mat}[1]{\operatorname{\mathsf{#1}}}
\newcommand{\cm}{\mathsf{m}} 
\newcommand{\cn}{\mathsf{n}}

\newcommand{\Gr}{\operatorname{Gr}}
\newcommand{\Grass}[2]{\Gr_{#1,#2}}
\newcommand{\mm}{m} 
\newcommand{\nn}{n}
\newcommand{\Grr}{\Grass{\mm}{\nn}} 
\newcommand{\grQCA}{C_q(\Grr)}

\renewcommand{\k}{\CC} 

\newcommand{\tgr}{*}
\newcommand{\dual}{^\vee}

\newcommand{\transp}{^{\mathsf{t}}}

\newcommand{\cluschar}[1]{\Psi_{#1}} 
\newcommand{\rkone}[1]{M_{#1}}
\newcommand{\maxNC}{\mathcal{S}}

\newcommand{\clugr}[1]{\eta_{#1}}
\newcommand{\minor}[1]{D_{#1}}
\newcommand{\qminor}[1]{\Delta_{#1}}

\newcommand{\vble}{q^{1/2}}
\newcommand{\basering}{\CC[q^{\pm 1/2}]}
\newcommand{\qbasering}{\CC[q, q^{-1}]}

\newcommand{\genfld}{\CC(\vble)}

\newcommand{\funE}{\mathbf{E}}
\newcommand{\funP}{\mathbf{P}}
\newcommand{\funJ}{\mathbf{J}}

\newcommand{\gabquiv}{\Gamma}
\newcommand{\mult}[2]{\left[#1 : #2 \right]}
\newcommand{\brakp}[1]{\left[#1\right]_+}
\newcommand{\abs}[1]{\bigl|{#1}\bigr|}
\newcommand{\om}{{b_{ik}}} 
\newcommand{\Tbs}{\overline{T}^*}

\newcommand{\dimvec}{\operatorname{\mathbf{dim}}}
\newcommand{\maxdiag}{\operatorname{MaxDiag}}
\newcommand{\mxd}[2]{\operatorname{\mathsf{MD}}(#1,#2)}
\newcommand{\valG}{\operatorname{val}_G}
\newcommand{\partn}[1]{\lambda_{#1}}

\begin{document} 

\begin{abstract}
In an earlier work, we gave an (additive) categorification of Grassmannian cluster algebras using the category 
$\CM(A)$ of Cohen-Macaulay modules for a certain Gorenstein order $A$.
In this paper, for each cluster tilting object in $\CM(A)$, 
we construct a compatible pair $(\mat{B}, \mat{L})$ in a way that is consistent with mutation.
This then determines a quantum cluster algebra and we show that, 
when $(\mat{B}, \mat{L})$ comes from a cluster tilting object with rank one summands, 
this quantum cluster algebra is (generically) isomorphic to the corresponding quantum Grassmannian. 
The key ingredient in the construction is a new invariant $\kappa(M,N)$ of modules $M,N$ in $\CM(A)$,
which also has an intriguing link with mirror symmetry.


\medskip
\noindent \emph{2020 Math. Subj. Class.} 
13F60 (primary), 14M15, 16G50, 20G42 (secondary). 
\end{abstract}

\maketitle
\section{Introduction}

Cluster algebras were created by Fomin and Zelevinsky \cite{FZ}. 
They are subalgebras of a freely-generated (or purely transcendental) field, 
generated by cluster variables obtained by mutations from an initial seed.  
Fomin and Zelevinksy also  showed in \cite{FZ} that the homogeneous 
coordinate rings $\CC[\Grass{2}{\nn}]$ of the simplest Grassmannians are 
cluster algebras (of finite type $A_{\nn-3}$). Scott \cite{Sc06} extended
this result to all Grassmannian coordinate rings $\CC[\Grr]$, although most 
are no longer of finite cluster type.

In \cite{JKS}, we gave an additive categorification of this cluster structure 
using the category $\CM(A)$ of Cohen-Macaulay modules for a certain 
algebra $A$, or equivalently, of equivariant Cohen-Macaulay modules for 
the plane curve singularity $x^\mm=y^{\nn-\mm}$. More precisely, in 
this categorification, clusters correspond to reachable cluster tilting objects
and cluster variables to reachable indecomposable rigid objects in this 
category. This work built on work of Geiss, Leclerc and Schr\"{o}er~\cite{GLS08} 
on categorification of the cluster structure on the affine coordinate 
rings of the open cells in Grassmannians and, much more generally, of unipotent 
subgroups of semisimple Lie groups.

Quantum cluster algebras were introduced by Berenstein and Zelevinsky \cite{BZ} as 
quantum analogues of the cluster algebras of Fomin and Zelevinsky. 
They are subalgebras of a \emph{quasi-free} skew field $\cF$, generated 
by cluster variables that are obtained by mutations from an initial quantum seed.
Note that by a quasi-free skew field we mean freely-generated by a finite set of 
quasi-commuting variables. Geiss, Leclerc and Schr\"{o}er partially generalised 
their results in \cite{GLS08} to the quantum setting in \cite{GLS13},
including the case of open cells in Grassmannians. By the delicate process of 
homogenising this result, Grabowski and Launois \cite{GL} proved that there is a 
quantum cluster algebra structure on the full Grassmannian quantum coordinate 
ring $\CC_q[\Grr]$, or at least on a small modification of it 
(see Theorem~\ref{thm:geniso-intro} below for more detail).

A crucial aspect of quantising a cluster algebra is that classical clusters become 
quasi-free generating sets for $\cF$. The quasi-commutation rules between these 
generators become additional data in the quantum seed,
namely a skew-symmetric matrix $\mat{L}$ that must be \emph{compatible} 
(Definition \ref{eq:compat}) with the exchange matrix $\mat{B}$.
Furthermore there is a corresponding mutation rule for $\mat{L}$.

Our first main result shows that such data $(\mat{B}, \mat{L})$
can be obtained as a natural invariant of 
the corresponding cluster tilting object in the category $\CM(A)$.  
Our approach is similar to \cite{GLS13}, but requires a new idea, because 
the homomorphism spaces in $\CM(A)$ are infinite dimensional. 
For any $M, N\in \CM(A)$, we consider the embedding
\[
\Hom_A(M, N)\to \Hom_Z(eM, eN),
\]
where $Z$ is the centre of $A$ and $e$ is a primitive idempotent in $A$ 
corresponding to a vertex of the quiver presenting $A$.  
The cokernel $K(M, N)$ of this embedding is in fact finite dimensional and the dimension 
\[
  \kappa(M, N)=\dim K(M, N)
\]
is the new invariant that we use to categorify the quantum seed.
More precisely, to a cluster tilting object $T=\oplus_i T_i$, we associate the standard exchange
matrix~$\mat{B}(T)$ corresponding to the Gabriel quiver of $\End(T)$ and a matrix $\mat{L}(T)=(\lambda_{ij})$ with 
\[
\lambda_{ij}=\kappa(T_j, T_i)-\kappa(T_i, T_j).
\] 
Note that all choices of the vertex are equivalent under a cyclic symmetry of $A$, and 
so this choice really just breaks that symmetry. 
A similar breaking of symmetry is necessary in the construction of the quantum Grassmannian. 

Using $\mu_k$ to denote mutation of  quantum cluster seeds as in 
\cite{BZ} and mutation of cluster tilting objects as \cite{BIRS}, we prove:

\begin{theorem}[Theorem~\ref{thm:main}]
The two matrices $\mat{B}$ and $\mat{L}$ associated to $T$ are compatible. 
Furthermore, mutation of cluster tilting objects is consistent with mutation of seed data,
i.e.~the pair associated to the mutated object $\mu_k(T)$
is the mutated pair $\mu_k(\mat{B}, \mat{L})$. 
\end{theorem}

This theorem means that our original category $\CM(A)$ provides consistent quasi-commutation rules for 
quantum clusters, by direct computation from cluster tilting objects. Note that we do 
not need to use a related category of graded objects with finite dimensional 
homomorphism spaces, as one could have expected. 

Next we show that these quasi-commutation rules give the familiar ones for the quantum minors.
Recall from \cite[\S5]{JKS} that $\CM(A)$ contains a collection of (rigid) rank one modules $M_I$, 
for each $\mm$-subset $I$ of $\{1,\ldots,\nn\}$, whose cluster characters are the classical minors 
(or Pl\"ucker coordinates) $\minor{I}$ in $\CC[\Grr]$. 
Furthermore, $\Ext^1(M_I, M_J)=0$ precisely when the labels $I,J$ 
are \emph{non-crossing} (Definition \ref{def:non-crossing}), and 
maximal non-crossing sets of labels correspond to cluster tilting objects in $\CM(A)$ with rank one summands.

On the other hand, the non-crossing condition is precisely the one for 
the quantum minors $\qminor{I}, \qminor{J}$ in $\CC_q[\Grr]$
to quasi-commute, as proved by Leclerc-Zelevinsky \cite{LZ98}, 
who also determined their quasi-commutation rule
\[
 \qminor{I}\qminor{J}=q^{c(I,J)}\qminor{J}\qminor{I},
\]
where $c(I,J)$ has an explicit combinatorial formula (see Definition \ref{def:non-crossing}).
Our second main result is thus to show the following.

\begin{theorem}[Theorem~\ref{thm:quasicomm}]
When $\qminor{I}$ and $\qminor{J}$ are quasi-commuting quantum minors,
\[
  c(I,J) = \kappa(M_J, M_I)-\kappa(M_I, M_J).
\]
\end{theorem}
As a consequence of this and combinatorial results of \cite{OPS}, there is a quantum cluster algebra
$C_q(\Grr)$, containing cluster variables $X_I$ for all $\mm$-subsets $I$ of $\{1,\ldots,\nn\}$, 
whose quasi-commutation rules match those of the quantum minors $\qminor{I}$.
By showing further that the appropriate quantum exchange relations are precisely quantum Pl\"ucker relations 
(Proposition~\ref{prop:newcv=minor}) 
we can deduce that $C_q(\Grr)$ contains the quantum coordinate ring $\CC_q[\Grr]$
in a way that identifies $X_I$ with $\qminor{I}$, for all $I$.

One would very much expect that these two algebras coincide in general, 
but this is known only when the cluster structure has finite type \cite{GLfin}.
What can be proved is that they conicide \emph{generically}, in following sense.
\begin{theorem}[Theorem~\ref{thm:generic-isom}]
\label{thm:geniso-intro}
There is an isomorphism 
\[
 \CC_q[\Grr]\otimes_{\basering} \genfld 
 \simeq C_q(\Grr) \otimes_{\basering}\genfld.
\]
\end{theorem}
This result is effectively what was proved in \cite{GL}, 
but we reprove it here in a more direct way, 
by applying the method of \cite{GLS13} in the homogeneous situation.

As a final note, we observe that our new invariant, in the case of rank one modules $M_I, M_J$, 
is equal to a known combinatorial invariant of
corresponding Young diagrams $\partn{I}$ and $\partn{J}$ 
(see \S\ref{sec:maxdiag} for details), namely
\begin{equation*}
  \kappa(M_I, M_J)=\maxdiag(\partn{I} \setminus \partn{J}).
\end{equation*}
The work of Rietsch-Williams \cite{RW17} on mirror symmetry for Grassmannians
shows that the invariant $\maxdiag$ appears when computing the 
Newton-Okounkov body associated to a valuation 
for a \emph{network torus} chart in the Grassmannian.
More precisely, expressing a Pl{\"u}cker coordinate $\minor{J}$ in these torus coordinates
gives a positive polynomial which can be described combinatorially 
using weights of flows on a network. 
The leading term in this \emph{flow polynomial} is determined by the quantities 
$\maxdiag(\partn{I} \setminus \partn{J})$,
for $I$ in a maximal non-crossing set of labels (see \cite[Thm.~1.2]{RW17}).

As observed in \cite[Remark~15.2]{RW17}, 
the invariant $\maxdiag$ also appears in another mirror symmetry context, namely the
work of Fulton-Woodward \cite{FW}, 
where it is used to compute the leading term in the product of two Schubert classes in
the quantum cohomology ring of the Grassmannian.

Combined with our results, it is becoming apparent that $\CM(A)$
plays a role in both the quantisation of and mirror symmetry for Grassmannians. 
The focus of this paper is on quantisation, 
while the connections to mirror symmetry will be the topic of a subsequent paper. 

The structure of this paper is as follows. In Sections~\ref{sec:qca}, 
\ref{sec:qgr} and \ref{sec:gcc}, 
we recall relevant background about quantum cluster algebras, quantum Grassmannians and 
the Grassmannian cluster category $\CM(A)$.
In Section~\ref{sec:maincon}, we introduce the new invariant $\kappa(M,N)$
and discuss various interpretations of it.
In Section~\ref{sec:seeds}, we use $\kappa(M, N)$ to construct quantum seed data from a cluster tilting object in $\CM(A)$.
We also show that this data correctly reproduces the quasi-commutation rules for quantum minors.
In Section~\ref{sec:maxdiag}, we explain the link with Rietsch-Williams' work on Grassmannian mirror symmetry.
In Section~\ref{sec:grCA}, we recall how the homogeneous coordinate ring of the Grassmannian becomes a cluster algebra
and, in Section~\ref{sec:grQCA}, we explain how far the arguments carry over to the quantum Grassmannian.
 
\section{Quantum cluster algebras} \label{sec:qca}

We recall those aspects of the definition of a quantum cluster algebra 
which are relevant for this paper. 
For a complete discussion, we
refer to the original work of Berenstein and Zelevinsky \cite{BZ}.
Let $\mat{B}=(b_{ij})$ be an $\cn\times \cm$ integer matrix, with $\cn\geq \cm$, and 
let $\mat{L}=(\lambda_{ij})$ be a skew-symmetric $\cn\times \cn$ integer matrix. 

\begin{remark}\label{rem:nondegen}
Note that these general cluster parameters $\cm,\cn$ should not be confused with the Grassmannian parameters $\mm,\nn$.
In fact, for the cluster algebra $\CC[\Grr]$, we have  
$\cn=\mm(\nn-\mm)+1$ and $\cn-\cm=\nn$.
Note also that, in order to have a non-trivial exchange matrix $\mat{B}$, i.e.~$\cm\geq 1$, 
and thus mutable cluster variables, we need to have $1<\mm<\nn-1$ and hence $\nn\geq 4$.
\end{remark}

\begin{definition}
We say that $\mat{B}$ and $\mat{L}$ are \emph{compatible} if 
\begin{equation}
\label{eq:compat}
\sum_{j=1}^{\cn} b_{jk}\lambda_{j\ell}=\delta_{k\ell}d_k
\end{equation}
for some positive integers $d_k$. 
In other words, the $\cm\times \cn$ matrix $\mat{B}\transp \mat{L}$ consists of two blocks: 
an $\cm\times \cm$ diagonal matrix $\mat{D}$
with diagonal entries $d_1, \dots, d_\cm$ followed by an $\cm\times (\cn-\cm)$ zero matrix. 
\end{definition}

The submatrix of $\mat{B}$ consisting of the first $\cm$ rows is called the \emph{principal part} of $\mat{B}$.
Note that, if a pair $(\mat{B}, \mat{L})$ is compatible, 
then $\mat{B}$ has full rank $m$ and its principal part is skew-symmetrizable \cite[Prop.~3.3]{BZ}. 

Given a compatible pair $(\mat{B}, \mat{L})$, choose $k$ in the \emph{mutable range} $1\leq k\leq \cm$ and 
define $\cn\times \cn$ and $\cm\times \cm$ matrices
$\mat{E}=(e_{ij})$ and $\mat{F}=(f_{ij})$ by 
\[
e_{ij}=
\begin{cases}
\delta_{ij} & \text{if $j\neq k$,}\\
-1 & \text{if $j=k=i$,}\\
\brakp{b_{ij}} & \text{if $j=k\neq i$,}
\end{cases}
\qquad
f_{ij}=
\begin{cases}
\delta_{ij} & \text{if $i\neq k$,}\\
-1 & \text{if $i=k=j$,}\\
\brakp{-b_{ij}} & \text{if $i=k\neq j$,}
\end{cases}
\]
where $\brakp{x}=\max\{0, x\}$.
Then the mutated matrices are defined to be
\begin{equation}
\label{eq:BLmutation}
\mu_k(\mat{L}) = \mat{E}\transp \mat{L} \mat{E}, \qquad \mu_k(\mat{B}) = \mat{E}\transp \mat{B} \mat{F}.
\end{equation}
The new pair $\mu_k(\mat{B}, \mat{L})$ is compatible \cite[Prop.~3.4]{BZ} 
and $\mu_k$ is an involution \cite[Prop.~3.6]{BZ}, that is,
$\mu_k\mu_k(\mat{B}, \mat{L}) = (\mat{B}, \mat{L})$.

Let $q$ be a formal variable, with a chosen square root $\vble$.
For brevity we write $\basering$ for the Laurent polynomial ring $\CC[q^{1/2},q^{-1/2}]$.

\begin{definition}\label{def:qtorus} 
The \emph{based quantum torus} $\cT(\mat{L})$ is the $\basering$-algebra 
generated by formal variables $X_1,\dots, X_\cn$ 
and their inverses $X_1^{-1},\dots, X_{\cn}^{-1}$, subject to the quasi-commutation relations 
\begin{equation}
\label{eq:qcomm}
X_iX_j=q^{\lambda_{ij}} X_jX_i.
\end{equation}
\end{definition}

For $\a=(a_i)\in \ZZ^{\cm}$, let 
\begin{equation}
\label{eq:qu-monom}
X^{\a}=q^{\gamma(\a)}X_1^{a_1}\dots X_{\cn}^{a_\cn}.
\end{equation}
where $\gamma(\a) = \sfrac12 \sum_{i>j}a_ia_j\lambda_{ij}$.
Then $\{X^{\a} \mid \a \in \ZZ^\cn\}$ is a $\basering$-basis of $\cT(\mat{L})$ and for $\a, \b\in \ZZ^\cn$, we have 
\[
X^{\a}X^{\b}=q^{ \sfrac12\sum_{i>j}(a_ib_j-b_ia_j)\lambda_{ij}}X^{\a+\b}
=q^{\sum_{i>j}(a_ib_j-b_ia_j)\lambda_{ij}}X^{\b}X^{\a}.
\]
Note that  $\cT(\mat{L})$ is an Ore domain \cite[Prop.~A.1]{BZ}, 
i.e.~it  can be embedded in its skew field of fractions $\cF(\mat{L})$. 

The datum $\Phi=( \{X_1, \dots, X_\cn\}, \mat{B}, \mat{L})$ forms what is called a 
\emph{quantum seed} (see \cite[Def.~4.5]{BZ} for a complete definition) and
the (indexed) set  $\{X_1, \dots, X_\cn\}$ is called a \emph{cluster}.
The mutation of pairs $(\mat{B},\mat{L})$ described above extends to an involutive operation on 
quantum seeds as follows.

For any $k$ in the mutable range $1\leq k\leq \cm$, define a new cluster variable
\begin{equation}
\label{eq:qu-exch}
X_k^*= X^{\a'}+X^{\a''},
\end{equation} 
where
\[
a_j'= \begin{cases}
  -1 & \text{if $j=k$,}\\
 \brakp{-b_{jk}} & \text{if $j\not=k$,}
\end{cases}
\quad\text{and}\quad
a_j''= \begin{cases}
-1 & \text{if $j=k$,}\\
 \brakp{b_{jk}} & \text{if $j\not=k$.}
\end{cases}
\]

\begin{remark}
\label{rem:compat}
Note that when $\ell\neq k$, we can rewrite \eqref{eq:compat} as
\begin{equation}
\label{eq:compat2}
\sum_{j:b_{jk}<0} (-b_{jk}) \lambda_{j\ell} =  \sum_{j:b_{jk}>0} b_{jk} \lambda_{j\ell} 
\end{equation}
This is precisely the condition that $X^{\a'}$ and $X^{\a''}$ have the same quasi-commutation rule with $X_\ell$,
and hence that $X_k^*$ quasi-commutes with $X_\ell$.
Indeed,
\[
 X_k^* X_\ell = q^{\lambda^*_{k\ell}} X_\ell X_k^* ,
 \]
where $\lambda^*_{k\ell}+\lambda_{k\ell}$ is the common value in \eqref{eq:compat2}.
Comparing this to \eqref{eq:BLmutation}, we see that $\mu_k(\mat{L})$ is precisely the matrix
$\mat{L}^*=(\lambda^*_{ij})$.
\end{remark}

Denote by $\mu_k \{X_1, \dots, X_\cn\}$ the set $\{X_1, \dots, X^*_k, \dots, X_\cn\}$, 
i.e.~$X_k$ is replaced by $X^*_k$ in the cluster $\{X_1, \dots, X_\cn\}$.  
The datum $\mu_k(\Phi)=\bigl( \mu_k\{X_1, \dots, X_\cn\}, \mu_k(\mat{B}, \mat{L}) \bigr)$ is 
called the \emph{mutation} of $\Phi$ in direction $k$,
and is again a quantum seed \cite[Prop.~4.7]{BZ} and 
thus $\mu_k\{X_1, \dots, X_\cn\}$ is another cluster in $\cF (\mat{L})$. 

Note that the initial variables $X_{\cm+1}, \dots, X_\cn$ are never mutated and appear in every cluster; they are called \emph{frozen (cluster) variables}.

\begin{definition}
\label{def:qca}
The quantum cluster algebra $C_q(\Phi)$ is the $\basering$-subalgebra of the skew field $\cF(\mat{L})$ generated by all the cluster variables appearing in all quantum seeds obtained from $\Phi$ by all possible sequences of mutations.
\end{definition}

Note that, by the quantum Laurent phenomenon \cite[\S5]{BZ}, we actually have $C_q(\Phi)\subseteq \cT(\mat{L})$.
Since $\cT(\mat{L})$ is a free $\basering$-module, $C_q(\Phi)$ is torsion free and therefore 
free, because $\basering$ is a principal ideal domain.

Recently, Geiss, Leclerc and Schr\"o{er} proved that, under suitable assumptions, the specialisation of $C_q(\Phi)$ at 
$\vble=1$ is the corresponding classical cluster algebra $C(X, \mat{B})$,
where $X$ stands for the cluster $\{X_1, \dots, X_\cn\}$
and $(X,\mat{B})$ is a classical seed.
Note that this is not an obvious result and is not just the statement that setting $\vble=1$ in the construction of $C_q(\Phi)$
gives $C(X, \mat{B})$. 

Recall from \cite{GLS18} that a $\ZZ$-graded cluster algebra is a $\ZZ$-graded algebra 
such that all the cluster variables are homogeneous.

\begin{theorem} \cite{GLS18} \label{thm:GLSflatdef}
Suppose that $C(X, \mat{B})$ is a $\ZZ$-graded cluster algebra with finite dimensional homogeneous components
and that $C(X, \mat{B})$ coincides with the corresponding upper cluster algebra $U(X, \mat{B})$. 
Then, for a quantum seed $\Phi=(X,\mat{B},\mat{L})$, we have
\[
  C_q(\Phi)\otimes_{\basering} \CC=C(X, \mat{B}),
\]
where $\CC$ here stands for the $\basering$-module $\basering / (\vble -1)$. 
\end{theorem}

\begin{remark}\label{rem:GLSflatdef}
Thus, under the assumptions of the theorem, $C_q(\Phi)$ is a flat deformation of $C(X, \mat{B})$.
Furthermore, $C_q(\Phi)$ will also be graded and each graded piece is a free $\basering$-module of finite rank, equal to the dimension of the corresponding graded piece of $C(X, \mat{B})$.
\end{remark}

\section{The Quantum Grassmannian} \label{sec:qgr}

In this section, we work over the deformation ground ring $\qbasering$ for simplicity, 
but we will later (in \S8) extend scalars to $\basering$, 
in order to make comparisons with cluster algebras. 

Denote by $\CC_q[M_{\mm\times \nn}]$ the \emph{quantum matrix algebra}, 
which is a $q$-deformation of the coordinate ring of the affine variety 
$M_{\mm\times \nn}$ of $\mm\times \nn$ complex matrices.
This is the graded $\qbasering$-algebra generated by degree $1$ variables $x_{ij}$, 
for $1\leq i \leq \mm$ and $1\leq j \leq \nn$, subject to the following homogeneous relations (cf.~\cite{FRT}):
\[
x_{ij}x_{st}=\begin{cases}
qx_{st}x_{ij} & \text{if $i=s$ and $j<t$;}\\
qx_{st}x_{ij} & \text{if $i<s$ and $j=t$;}\\
x_{st}x_{ij} & \text{if $i<s$ and $j>t$;}\\
x_{st}x_{ij}+(q-q^{-1})x_{sj}x_{it} & \text{if $i<s$ and $j<t$.}
\end{cases} 
\]
This algebra is free as a $\qbasering$-module, with finite rank graded pieces. 
Indeed, it has a basis consisting of reverse-lexicographically ordered monomials
and the above relations provide a straightening law. 
Note that, unlike the coordinate ring on the variety of matrices, the quantum matrix algebra
depends on the choice of a total order on the rows and columns. 

Any $I\subseteq \{1, \dots, \nn\}$ with $|I|=\mm$ determines an $\mm\times \mm$ quantum minor 
\[
\qminor{I}=\sum_{\sigma\in S_\mm}(-q)^{l(\sigma)} x_{1i_{\sigma(1)}}\dots x_{\mm i_{\sigma(\mm)}},
\]
where $S_\mm$ is the symmetric group of $\{1, \dots, \mm\}$ and $l(\sigma)$ is the length of $\sigma\in S_\mm$. 
Setting $q=1$, these become the ordinary minors that generate the coordinate ring $\CC[\Grr]$ 
of the Grassmannian $\Grass{\mm}{\nn}$ of $\mm$-dimensional quotients of $\CC^\nn$.

\begin{definition}\cite{LL}
\label{def:qgrass}
The coordinate ring $\CC_q[\Grr]$  of the \emph{quantum Grassmanniann} 
is the $\qbasering$-subalgebra of $\CC_q[M_{\mm\times \nn}]$ 
generated by all $\mm\times \mm$ quantum minors.
Thus it has an induced (and scaled) grading in which each quantum minor has degree~1.
\end{definition}

Kelly, Lenagan and Rigal  \cite[Thm 2.7]{KLR} constructed a homogeneous basis for $\CC_q[\Grr]$ 
which specialises to a standard basis of $\CC[\Grr]$ on setting $q=1$.
As an immediate consequence we have the following.
\begin{theorem} \label{thm:defgr}
The quantum Grassmannian $\CC_q[\Grr]$ is a flat deformation of the classical coordinate ring $\CC[\Grr]$, 
which is the specialisation at $q=1$.
Each graded piece is a free $\qbasering$-module of finite rank, equal to the dimension of the corresponding graded piece of $\CC[\Grr]$.
\end{theorem}

Two quantum minors $\qminor{I}$ and $\qminor{J}$ are said to \emph{quasi-commute} if 
\[
  \qminor{I}\qminor{J}=q^c\qminor{J}\qminor{I},
\]
for some integer $c$. 
Leclerc and Zelevinsky \cite{LZ98} describe the quasi-commuting quantum minors and give a combinatorial 
formula for computing the exponent $c$, as follows. 

\begin{definition}\label{def:non-crossing}
Two $\mm$-element subsets $I$ and $J$ of $\{1, \dots, \nn\}$ are said to be 
\emph{non-crossing} (or \emph{weakly separated)} if 
one of the following two conditions holds: 
\begin{itemize}
\item[(i)] $J\setminus {I}$ can be written as a disjoint union $J'\cup J''$ so that $J'< (I\setminus J) <J''$;
\item[(ii)] $I\setminus {J}$ can be written as a disjoint union $I'\cup I''$ so that $I'< (J\setminus I) <I''$,
\end{itemize}
where $I<J$ means that $i<j$ for all $i\in I$ and $j\in J$.
When $I$ and $J$ are non-crossing, we define 
\[
  c(I,J) = \begin{cases} 
 |J''| - |J'| & \text{in case (i)}\\
 |I'| - |I''| & \text{in case (ii)}
\end{cases}
\]
When both conditions hold, these two formulae give the same answer.
\end{definition}

\begin{theorem}\label{thm:LZ}\cite{LZ98}
Two quantum minors $\qminor{I}$ and $\qminor{J}$ quasi-commute if and only if $I$ and $J$ are non-crossing. 
In this case, the quasi-commutation rule is
\[
 \qminor{I}\qminor{J}=q^{c(I,J)}\qminor{J}\qminor{I}.
\]
\end{theorem}

\begin{remark}\label{rem:cycsym}
Note that, despite the formulation in Definition~\ref{def:non-crossing}, 
the non-crossing condition depends only on the cyclic order the index set $\{1,\ldots,\nn\}$. 
On the other hand, the definitions of $c(I,J)$ and of the quantum Grassmannian 
depend crucially on the choice of total order of the index set,  
as made in the definition of the quantum matrix algebra. 
\end{remark}

\section{The Grassmannian cluster category} \label{sec:gcc}

In this section, we recall from \cite{JKS} how the Grassmannian cluster algebra has a categorification given by the category $\CM(A)$ of Cohen-Macaulay modules for a suitable algebra $A$.

Consider the commutative ring
\begin{equation}\label{eq:Rdef} 
  R=\k[[x,y]]/(x^\mm-y^{\nn-\mm})
\end{equation}
together with the action of $G=\{\zeta\in \k^* : \zeta^\nn=1\}$, via 
\begin{equation}\label{eq:Gaction}
  (x,y)\mapsto (\zeta x, \zeta^{-1} y).
\end{equation}
The category $\md_G(R)$ of finitely generated $G$-equivariant $R$-modules is tautologically equivalent to the 
finitely generated module category $\md(A)$
for the twisted group ring $A=R\tgr G$.
The centre of $A$ is the $G$-invariant subring of $R$, namely
\begin{equation}\label{eq:Zdef} 
Z=\k[[t]],
\end{equation}
where $t=xy$.
Note that $R$ and $A$ are free $Z$-modules of rank $\nn$ and $\nn^2$, respectively.
Furthermore, $R$ is a Gorenstein ring, while $A$ is a non-commutative $Z$-order,
which is also Gorenstein, in the sense of Buchweitz (cf.~\cite[Cor. 3.7]{JKS}).

Exploiting the identification $\md(A)=\md_G(R)$, we define
\begin{equation}\label{eq:CM2}
  \CM(A) = \CM_G(R),
\end{equation}
that is, the category of $G$-equivariant Cohen-Macaulay $R$-modules.
This is a Frobenius category, in the sense of \cite{Hap}, as $R$ is Gorenstein.
Furthermore, it is stably 2-Calabi-Yau,
that is, $\Ext^1(M,N)$ is naturally dual to $\Ext^1(N,M)$,
for all $M,N$ in $\CM(A)$.
As $R$ is a finitely generated $Z$-module and $Z$ is a principal ideal domain, the Cohen-Macaulay $R$-modules
are precisely those which are free over $Z$ and thus $\CM(A)$ may also be described directly as
the category of finitely generated $A$-modules which are free over $Z$.

Another way to describe the algebra $A$ is as a quotient of the 
complete path algebra $\widehat{\k Q}$ of a quiver $Q$, 
which is the doubled quiver of a simple circular graph $C$.
More precisely, let $C=(C_0,C_1)$ be the circular graph with vertex set 
$C_0=\ZZ_\nn=G\dual$ and edge set $C_1=\{1,\ldots ,\nn\}$, with edge $i$
joining vertices $(i-1)$ and $(i)$. 
The associated quiver $Q=Q(C)$ has vertex set $Q_0=C_0$ and
arrows set $Q_1=\{x_a,y_a: a\in C_1\}$ with $x_a\colon (i-1) \to (i)$
and $y_a\colon (i) \to (i-1)$, as illustrated in Figure~\ref{fig:mckayQ} in the case $\nn=5$.

As is familiar from the McKay correspondence, $\k[[x,y]]\tgr G$ is isomorphic to the
complete preprojective algebra of type $\widetilde{A}_{\nn-1}$, that is, the quotient of $\widehat{\k Q}$
by the ideal generated by the $\nn$ relations $xy=yx$, one beginning at each vertex.
If we quotient further by the $\nn$ relations $x^\mm=y^{\nn-\mm}$,
then we obtain $A=R\tgr G$.

The quotient of $A$ by the ideal generated by the idempotent $e_0$ at vertex~$0$ 
is the preprojective algebra $\Pi$ of type $A_{\nn-1}$.  
The proof in \cite{JKS} that $\CM(A)$ categorifies the cluster structure on the Grassmannian 
uses the quotient functor 
\begin{equation}\label{eq:pi-def-old} 
\pi\colon \CM(A) \to \md\Pi,
\end{equation}
whose image is the subcategory $\Sub Q_\mm$ of modules with socle at $\mm$,
and the result of Geiss-Leclerc-Schr\"oer \cite{GLS08} that $\Sub Q_\mm$ gives a categorification for the
open cell in the Grassmannian.  

Note that there is actually nothing special about the vertex 0, as the algebra $A$ has a cyclic symmetry 
induced by cycling the graph $C$. 
Thus any other vertex could have been chosen instead.

\begin{figure}
\begin{tikzpicture}[scale=1,baseline=(bb.base)]  
\coordinate (bb) at (0,0); 
\newcommand{\circradius}{1.5cm}
\newcommand{\inradius}{1.3cm}
\newcommand{\outradius}{1.8cm}
\draw[blue,thick] (0,0) circle(\circradius);
\foreach \j/\v in {1/1,2/2,3/3,4/4,5/0}
{\draw (90-72*\j:\circradius) node[black] {$\bullet$};
 \draw (90-72*\j:\outradius) node[black] {$\v$};
 \draw (126-72*\j:\inradius) node[black] {$\j$}; }
\end{tikzpicture}
\qquad\qquad
\begin{tikzpicture}[scale=1,baseline=(bb.base)]
\coordinate (bb) at (0,0); 
\newcommand{\radius}{1.5cm}
\foreach \j in {1,...,5}{
  \path (90-72*\j:\radius) node[black] (w\j) {$\bullet$};
  \path (162-72*\j:\radius) node[black] (v\j) {};
  \path[->,>=latex] (v\j) edge[blue,bend left=25,thick] node[black,auto] {$x_{\j}$} (w\j);
  \path[->,>=latex] (w\j) edge[blue,bend left=25,thick] node[black,auto] {$y_{\j}$}(v\j);
}
\draw (90:\radius) node[above=3pt] {$0$};
\draw (162:\radius) node[above left] {$4$};
\draw (234:\radius) node[below left] {$3$};
\draw (306:\radius) node[below right] {$2$};
\draw (18:\radius) node[above right] {$1$};
\end{tikzpicture}
\caption{The circular graph $C$ and double quiver $Q(C)$}
\label{fig:mckayQ}
\end{figure}

For any $A$-module $M$ we can define its \emph{rank} (see \cite[Def.~3.5]{JKS})
\begin{equation}\label{eq:rank}
 \rk(M) = \len_{A\otimes_Z K}\bigl( M\otimes_Z K \bigr), 
\end{equation}
where $K$ is the field of fractions of $Z$ and we note that $A\otimes_Z K\isom M_n\bigl( K\bigr)$.
For any $M$ in $\CM(A)$ and every $j\in Q_0$, we have
\[
  \rk_Z(e_j M) = \rk(M).
\]
In other words, every such $M$ may be regarded as a representation of the quiver $Q$, 
with a free $Z$-module of rank $\rk(M)$ at each vertex and satisfying the relations $xy=t=yx$ and 
$x^\mm=y^{\nn-\mm}$.

In particular, it is possible \cite[Prop 5.2]{JKS} to classify the rank one modules in $\CM(A)$,
as follows.
For any $\mm$-subset $I\sub C_1$, define $\rkone{I}$ in $\CM(A)$ as follows.
For $j\in Q_0$, set $e_j\rkone{I}=Z$ and, for $a\in C_1$, set 
\begin{align*}
 & \text{ $x_a\colon Z\to Z$ to be (multiplication by) $1$, if $a\in I$, or $t$, if $a\not\in I$,}\\
 & \text{ $y_a\colon Z\to Z$ to be (multiplication by) $t$, if $a\in I$, or $1$, if $a\not\in I$.}
\end{align*}
Thus the rank one modules are in canonical one-one correspondence with the Pl\"ucker coordinates
$\minor{I}$ in $\CC[\Grr]$ and indeed these are the cluster characters of the corresponding modules
\cite[\S9]{JKS}.

A key point of the categorification is the following result \cite[Prop 5.6]{JKS}

\begin{proposition}\label{prop:nc}
Let $I,J$ be $\mm$-subsets of $C_1$.
Then $\Ext^1_A(\rkone{I},\rkone{J})=0$ if and only if $I$ and $J$ are non-crossing.
\end{proposition}

This explains why the non-crossing condition should be invariant under cycling the indices
 (cf.~Remark~\ref{rem:cycsym}), since the Ext-vanishing condition clearly is, due to the cyclic symmetry of $A$.

It follows from Proposition~\ref{prop:nc} that, 
for any maximal non-crossing collection~$\maxNC$, 
\begin{equation}\label{eq:TmaxNC}
 T_\maxNC = \bigoplus_{J\in\maxNC} \rkone{J}
\end{equation}
is a cluster tilting object in $\CM(A)$ \cite[Rem 5.7]{JKS}.
Note that in our context cluster tilting objects are the same as maximal rigid objects
\cite[Thm II.1.8]{BIRS} (see also \cite[Rem 4.8]{JKS}).
One important additional property is the following.

\begin{proposition}\label{2cycles}
If $\nn\geq 3$, then, for any basic maximal rigid module $T$ in $\CM(A)$, 
the Gabriel quiver $\gabquiv_T$ of the algebra $\End_A(T)$ has no loops or $2$-cycles. 
\end{proposition}

\begin{proof}
Write $T=P_0\oplus T'$. 
The Gabriel quiver $Q'$ of $\End(\pi T')$ in $\Sub Q_\mm$ has 
no loops or $2$-cycles, 
by the arguments in \cite[\S8.1]{GLS08} and \cite[Thm 2.2]{GLS06b}.
On the other hand, $Q'$ is obtained from the Gabriel quiver $Q$ of $\End_A(T)$ 
by deleting the vertex $0$ and all incident arrows, 
so $Q$ has no loops or $2$-cycles at any vertices different from $0$. 
By the cyclic symmetry of $A$, the same argument applies when we replace $P_0$ by 
any other indecomposable projective-injective $P_j$ in $\CM(A)$.
Now, provided $\nn\geq 3$, there is a vertex $j$ of $\gabquiv_T$ (corresponding to $P_j$) 
different from any given pair of vertices of $\gabquiv_T$ and so the result follows.
\end{proof}

\section{The invariant $\kappa(M,N)$} \label{sec:maincon}

To make the definition, we need to choose a vertex in $Q_0$. 
Any vertex will do, due to the cyclic symmetry of $A$, 
but this choice then breaks the cyclic symmetry 
and gives a total order on the index set $C_1=\{1,\ldots,\nn\}$.
Hence, for simplicity, we choose the vertex $0$,
giving the usual order on $C_1$, which has already been used in defining the quantum 
Grassmannian in Section~\ref{sec:qgr} (cf.~Remark~\ref{rem:cycsym}).
In Section~\ref{sec:seeds} we will see the direct relationship between this choice and the quantum structure.

For $M\in\CM(A)$, we write $M_0$ for the free $Z$-module $e_0 M$. 
The restriction functor
\begin{equation}\label{eq:E}
\funE\colon \md A \to \md Z\colon M\mapsto M_0
\end{equation}
may also be written as
\[
 \funE =\Hom_A(Ae_0,-)=e_0A\otimes_A -
\] 
and has left and right adjoints $\funP,~ \funJ \colon \md Z \to \md A$ given by
\[
 \funP= Ae_0 \otimes_Z -, \qquad \funJ = \Hom_Z(e_0A,-).
\] 
When $W$ is a free $Z$-module, both $\funP W $ and $\funJ W $ are projective $A$-modules
(\cite[Lemma 3.6]{JKS}).
Since $e_0Ae_0=Z$ we also have canonical identifications
\begin{equation}\label{eq:can-iso}
 M_0 = \funE\funP M_0 \text{ and }  \funE\funJ N_0 = N_0
\end{equation}
given by the unit/counit of the adjunction.

Consequently, the natural adjunction isomorphisms 
\begin{equation}\label{eq:nat-iso}
 \Hom_A(\funP M_0, N) \isoto \Hom_Z(M_0, N_0) \isofrom \Hom_A(M, \funJ N_0)
\end{equation}
may simply be described as applying the functor $\funE$ and making the canonical identification.

\begin{lemma}\label{lem:pi-om}
When $M,N\in\CM(A)$, the natural maps $\alpha\colon \funP M_0 \to M$ and $\beta\colon N\to  \funJ N_0$ are injective with finite dimensional cokernels
$\pi M$ and $\omega N$, respectively, 
i.e.~we have short exact sequences
\begin{align}
 0\to \funP M_0  \stackrel{\alpha}{\to} M & \to \pi M \to 0 
\label{eq:pi-def}\\
 0\to N\stackrel{\beta}{\to}  \funJ N_0  & \to \omega N \to 0
\label{eq:om-def}
\end{align}
\end{lemma}
\begin{proof}
In both cases, the restriction of the natural map to the vertex $0$ is the identity map, so the kernel has rank 0
and hence, being in $\CM(A)$, it must be zero.
The cokernel also has rank 0 and is finitely generated, so is finite dimensional. 
\end{proof}

\begin{remark}\label{rem:pi-om}
Note that $\pi$ and $\omega$ are both (exact) functors $\CM(A)\to \md\Pi$.
Indeed, $\pi$ is the same as in \eqref{eq:pi-def-old} and has image the subcategory $\Sub Q_{\mm}$, 
while $\omega$ has image the subcategory $\Fac Q_{\mm}$ of modules with top at $\nn-\mm$.
Both $\Sub Q_{\mm}$ and $\Fac Q_{\mm}$ give 
categorifications for open cells in the Grassmannian, so we could have chosen $\omega$ instead of $\pi$, 
when proving, in \cite{JKS}, that $\CM(A)$ is a categorification of the Grassmannian
(cf.~\cite[Sec~3.2]{GLS08}).
\end{remark}

\begin{definition}\label{def:Kv}
For $M,N \in\CM(A)$, define  the map 
\[ \phi_0\colon \Hom_A(M, N)\to\Hom_Z(M_0, N_0)\colon f\mapsto \funE(f) \]
and let
\[ K(M, N)=\coker\phi_0. \]
\end{definition}

We have the following equivalent descriptions of $\phi_0$,
and thus of its cokernel.
 
\begin{lemma}\label{lem:phi-equivs}
We have the following commutative diagram. 
\[
\begin{tikzpicture}[scale=0.8]
\draw (6, 1) node(a2) {$\Hom_A(M, \funJ N_0)$};
\draw (1,3) node (b1) {$\Hom_A(M, N)$};
\draw (6,3) node (b2) {$\Hom_Z(M_0, N_0)$};
\draw (11,3) node(b3) {$\Hom_A(\funP M_0, \funJ N_0)$,};
\draw (6,5) node(c2) {$\Hom_A(\funP M_0, N)$};
\draw [cdmap] (b1) to node [above]{\maplab $\phi_0$} (b2);
\draw [cdmap] (b3) to node [above] {\maplab $\isom$}(b2);
\draw [cdmap] (b1) to node [below left] {\maplab $\beta_*(M)$} (a2);
\draw [cdmap] (b1) to node [above left] {\maplab $\alpha^*(N)$} (c2);
\draw [cdmap] (a2) to  node [below right] {\maplab $\alpha^*(\funJ N_0)$} (b3);
\draw [cdmap] (c2) to node [above right] {\maplab $\beta_*(\funP M_0)$} (b3);
\draw [cdmap] (a2) to node [right] {\maplab $\isom$} (b2);
\draw [cdmap] (c2) to node [right] {\maplab $\isom$} (b2);
\end{tikzpicture}
\]
where the four inward maps are all given by applying the functor $\funE$ and making the canonical identifications in \eqref{eq:can-iso}.

Thus, $K(M, N)$ is also naturally the cokernel of any one of the three injective maps $\alpha^*(N)$,
$\beta_*(M)$ and $\beta_*(\funP M_0) \circ \alpha^*(N) = \alpha^*(\funJ N_0) \circ \beta_*(M)$.
\end{lemma}

\begin{proof}
The vertical isomorphisms are precisely those of \eqref{eq:nat-iso}
and the third unnamed isomorphism can be identified with either of two adjunction isomorphisms
using \eqref{eq:can-iso}.
The fact that the four triangles commute follows immediately from the fact that the natural maps $\alpha$ and $\beta$ restrict to the identity map under $\funE$.
\end{proof}

\begin{lemma}\label{lem:Kv}
Let $M,N \in\CM(A)$. Then  the map 
\[
\phi_0\colon \Hom_A(M, N)\to\Hom_Z(M_0, N_0)
\]
 is injective, i.e.~we have the short exact sequence 
\begin{equation}
\label{eq:def-Kv}
0\to  \Hom_A(M, N)\to  \Hom_Z(M_0, N_0)\to K(M, N)\to 0.
\end{equation}
Further, the cokernel $K(M, N)$ is finite dimensional.
\end{lemma}
 
\begin{proof}
By Lemma~\ref{lem:phi-equivs},
we can identify the map $\phi_0$ with the map $\beta_*(M)$
obtained by applying $\Hom_A(M,-)$ to the injective map $N\to  J N_0$. 
Hence $\phi_0$ is injective.
Applying $\Hom_A(M,-)$ to the whole short exact sequence \eqref{eq:om-def}, 
we can also identify $K(M, N)$ with the image of the induced map $\Hom_A(M, \funJ N_0)\to \Hom_A(M,\omega N)$ 
and hence it is finite dimensional, because $\omega N$ is, by Lemma ~\ref{lem:pi-om}.
\end{proof}

\begin{definition} \label{def:kappa}
We define 
\[ \kappa(M, N) =\dim  K(M, N), \]
where $ K(M, N)$ is as in Definition~\ref{def:Kv} and is finite dimensional by Lemma~\ref{lem:Kv}.
\end{definition}

\begin{lemma} \label{lem:commlseq}
The quotient map $p\colon M\to \pi M$ in \eqref{eq:pi-def} induces an isomorphism 
$p^*(\omega N) \colon \Hom_A(\pi M, \omega N) \isoto \Hom_A(M, \omega N)$.
Thus we obtain an isomorphism of long exact sequences
\[
\begin{tikzpicture}[xscale=0.85, yscale=0.8]
\draw (-1.7, 3) node (b0) {$\cdzero$};
\draw (0.7,3) node (b1) {$\Hom_A(M, N)$};
\draw (4.8,3) node (b2) {$\Hom_A(M, \funJ N_0)$};
\draw (9.3,3) node (b3) {$\Hom_A(M, \omega N)$};
\draw (13, 3) node (b4) {$\Ext_A^1(M, N)$};
\draw (15.2,3) node (by) {$\cdzero$};
\draw(-1.7, 1) node(c0) {$\cdzero$};
\draw (0.7,1) node (c1) {$\Hom_A(M, N)$};
\draw (4.8,1) node (c2) {$\Hom_Z(M_0, N_0) $}; 
\draw (9.3, 1) node(c3) {$\Hom_A(\pi M, \omega N)$};
\draw (13, 1) node(c4) {$\Ext_A^1(M, N)$};
\draw (15.2,1) node (cy) {$\cdzero$};
\foreach \T/\H in {b0/b1, b1/b2, b2/b3, b3/b4, b4/by, c0/c1, c4/cy}
  \draw [cdmap] (\T) to (\H);
\draw [cdmap] (c1) to node [above] {\maplab $\phi_0$} (c2) ;
\draw [cdmap] (c2) to node [above] {\maplab $\psi$}(c3);
\draw [cdmap] (c3) to node [above] {\maplab $\delta$} (c4);
\draw [equals] (b1) to (c1); 
\draw [cdmap] (b2) to node [right] {\maplab $\isom$} (c2);
\draw [cdmap] (b3) to node [right] {\maplab $\isom$} (c3);
\draw [equals] (b4) to (c4); 
\end{tikzpicture}
\]
where $\psi$ and $\delta$ are defined to make the middle and right-hand squares commute,
and we may identify $K(M, N)$ with the image of $\psi$ in $\Hom(\pi M, \omega N)$.
\end{lemma}

\begin{proof}
First, as $\CM(A)$ is a Frobenius category and $\funJ N_0$ is projective-injective, 
\[
\Ext_A^1(M, \funJ N_0)=0.
\] 
So applying $\Hom_A(M, -)$ to \eqref{eq:om-def} gives  the top long exact sequence.
Next, $p^*(\omega N)$ is an isomorphism, 
as $\funE \omega N=0$, so $\Hom_A(\funP M_0, \omega N)=0$.
Finally, the first vertical isomorphism is from Lemma~\ref{lem:phi-equivs},
while the second is $p^*(\omega N)^{-1}$.
\end{proof}

\begin{remark}
If we apply $\Hom_A(-,N)$ to \eqref{eq:pi-def} and note that
$\Ext_A^1(\funP M_0,N)=0$, as $\funP M_0$ is a projective module, 
and $\Hom_A(\pi M, N)=0$, as $\pi M$ is finite dimensional and $N\in \CM A$,
then we get another long exact sequence
\[ 
  0\to  \Hom_A(M, N)\to  \Hom_A(\funP M_0,N)\to \Ext_A^1(\pi M, N)\to \Ext_A^1(M,N) \to 0.
\]
In fact, this sequence is isomorphic to the two sequences in Lemma \ref{lem:commlseq}.
In particular, the identification of the middle maps rests on the fact that a map of short exact sequences 
induces an equivalence between the short exact sequences obtained by pull-back 
and push-forward respectively.
\end{remark}

\begin{remark}
Recall \cite[Cor.~4.6]{JKS} that $\pi$, as in \eqref{eq:pi-def-old}, 
induces an isomorphism $\Ext^1_A(M,N)\isom\Ext^1_{\Pi}(\pi M, \pi N)$.
Therefore, the map $\delta$ in Lemma~\ref{lem:commlseq} gives a possible description
of the invariant $\kappa(M,N)$ within $\md \Pi$.
However, this description uses both the representations $\pi M, \pi N$ from $\Sub Q_\mm$
and $\omega N$ from $\Fac Q_\mm$.
\end{remark}

\section{Constructing seed data from a cluster tilting object} \label{sec:seeds}

In this section, we construct two matrices $\mat{B}=\mat{B}(T)$, in \eqref{eq:BfromT}, and $\mat{L}=\mat{L}(T)$, in \eqref{eq:LfromT},
associated to a cluster tilting object $T$ in $\CM(A)$. 
We show that they are compatible and mutate consistently when $T$ mutates. 
The definition of $\mat{B}$ is standard, but that of $\mat{L}$ uses our new invariant $\kappa$ from the previous section.
We also show that, when $T=T_\maxNC$ as in \eqref{eq:TmaxNC}, 
the matrix $\mat{L}$ recovers the quasi-commutation rules for the corresponding quantum minors
(see Theorem~\ref{thm:LZ}).

Let $T=\oplus_{i=1}^{\cn} T_i\in \CM(A)$ be a cluster tilting object, 
where each $T_i$ is indecomposable and the last $\cn-\cm$ summands 
are the indecomposable projective objects in $\CM(A)$. 
For any two summands of $T$, we write 
\[
  \Irr(T_i, T_j)=\rad(T_i, T_j)/\rad^2(T_i, T_j)
\]
for the `space of irreducible maps' $T_i\to T_j$ in $\add T$ 
and recall that $\dim \Irr(T_i, T_j)$ is the number of arrows $i\to j$ 
in the Gabriel quiver $\gabquiv_T$ of $\End_A(T)$.
Of course, strictly speaking, the irreducible maps are those in 
$\rad(T_i, T_j) \setminus \rad^2(T_i, T_j)$.

Let $\mat{B}=(b_{ij})$ be the $\cn\times \cm$-matrix defined by 
\begin{equation}
\label{eq:BfromT}
  b_{ij} = \dim \Irr(T_i, T_j)-\dim \Irr(T_j, T_i),
\end{equation}
where $1\leq i\leq \cn$ and $1\leq j\leq \cm$.
Note that, if $\mat{B}$ is non-trivial, i.e. $\cm\geq 1$, 
then by Remark~\ref{rem:nondegen}, we must have $\nn\geq 4$.
Hence, by Proposition~\ref{2cycles}, at least one of $\dim \Irr(T_j, T_i) $ and $\dim \Irr(T_i, T_j)$ is zero. 
Thus $b_{ij}\geq 0$ if and only if there are no arrows in $\gabquiv_T$ from $j$ to $i$,
and then $b_{ij}$ is the number of arrows from $i$ to $j$.
Similarly, $b_{ij}\leq0$ if and only if there are no arrows in $\gabquiv_T$ from $i$ to $j$,
and then $-b_{ij}$ is the number of arrows from $j$ to $i$.

Suppose that $1\leq k\leq \cm$, so that $T_k$ is a non-projective indecomposable summand of $T$. 
Then (cf.~\cite[\S5]{GLS06b}, in particular Propositions~5.6 and 5.7) we have the following short exact \emph{exchange sequences} for $T_k$: 
\begin{align}
&0\to T_k \longto T'_k
 \longto T_k^*\to 0, \label{eq:right-mutation} \\
&0\to T^*_k \longto T''_k
 \longto T_k \to 0, \label{eq:left-mutation}
\end{align}
where
\begin{equation*}
T'_k = \bigoplus_{j:b_{jk}<0} T_j^{-b_{jk}}
\quad\text{and}\quad
T''_k = \bigoplus_{j:b_{jk}>0} T_j^{b_{jk}}.
\end{equation*}
Note that the maps in these sequences are minimal left/right approximations,
and hence are made up of irreducible maps (inducing a basis of $\Irr$) to and from $T_k$ or $T_k^*$.
The new cluster tilting object $(T/T_k)\oplus T^*_k$ 
is called the \emph{mutation} of $T$ in direction $k$ and denoted by $\mu_k(T)$. 

Mutation of a cluster tilting object $T$ is compatible with mutation of the matrix $\mat{B}(T)$. 
For the convenience of the reader, we provide a proof, 
adapted from the proof in \cite[Theorem II.1.6]{BIRS} to our setting with minor modifications.

\begin{theorem} \label{thm:mutationmatrix} 
Let $T$ be a cluster tilting object and $T_k$ be a non-projective indecomposable summand. 
Then
\[
  \mat{B}(\mu_k(T))=\mu_k(\mat{B}(T)).
\]
\end{theorem}

\begin{proof}
The exchange sequences (\ref{eq:right-mutation}) and (\ref{eq:left-mutation}) provide
an orientation-reversing correspondence between the arrows incident to vertex $k$ in the Gabriel quivers of 
$\End_A(T)$ and $\End_A(\mu_k(T))$. 
Hence
\[
  \mat{B}(\mu_k(T))_{ik}=-b_{ik}=\mu_k(\mat{B}(T))_{ik}.
\]
Thus it remains to consider the entries $\mat{B}(\mu_k(T))_{ji}$, where $i,j\neq k$. 
By definition (of~$\mat{B}$), we know that $T_i$ is not projective.
We consider in detail the case where there is no arrow from $k$ to $i$
or equivalently $b_{ik}\geq 0$.
Thus  $b_{ik}$ is the number of arrows from $i$ to $k$.

Consider the exchange sequence \eqref{eq:right-mutation} for $T_i$ instead of $T_k$
and write the middle term $T_i'=D_i'\oplus T_k^{\om}$.
Combined with the sequence \eqref{eq:right-mutation} for $T_k$, we get
the following commutative diagram of four short exact sequences
\[
\begin{tikzpicture}[xscale=2.3, yscale=1.6]
\pgfmathsetmacro{\hoff}{0.55}
\pgfmathsetmacro{\voff}{0.7}
\draw (1,1) node (A1) {$T_i$};
\draw (2,1) node (A2) {$D_i'\oplus T_k^{\om}$};
\draw (3,1) node (A3) {$T^*_i$};
\draw (1,0) node (B1) {$T_i$};
\draw (2,0) node (B2) {$ D_i'\oplus (T'_k)^\om$};
\draw (3,0) node (B3) {$X$};
\draw (2,-1) node (C2) {$(T_k^*)^\om$};
\draw (3,-1) node (C3) {$(T_k^*)^\om$};
\draw (A1)++(-\hoff,0) node (a1) {$\cdzero$};
\draw (A3)++(\hoff,0) node (a3) {$\cdzero$};
\draw (B1)++(-\hoff,0) node (b1) {$\cdzero$};
\draw (B3)++(\hoff,0) node (b3) {$\cdzero$};
\draw (A2)++(0,\voff) node (aa2) {$0$};
\draw (A3)++(0,\voff) node (aa3) {$0$};
\draw (C2)++(0,-\voff) node (cc2) {$0$};
\draw (C3)++(0,-\voff) node (cc3) {$0$};
\foreach \P/\Q in {a1/A1, A1/A2, A2/A3, A3/a3, b1/B1, B2/B3, B3/b3,
  B2/C2, A2/B2, B3/C3, aa2/A2, aa3/A3, C2/cc2, C3/cc3}
  \draw [cdmap] (\P) to (\Q);
\draw [cdmap] (A3) to node [right] {\maplab $g$} (B3);
\draw [cdmap] (B1) to node [above] {\maplab $a$} (B2);
\foreach \P/\Q in {A1/B1, C2/C3}
  \draw [equals] (\P) to (\Q);
\end{tikzpicture}
\] 
Now $a$ is the composition of the two left approximations in the diagram
and, by construction, neither $D_i'$ nor $T_k'$ have any summands isomorphic to $T_k$ or $T_i$.
In particular, there is no $T_i$ summand in $T_k'$ because there is no arrow $k\to i$.
Thus $a$ is a left $\add T/(T_k\oplus T_i)$-approximation.

However, in order to compute $\mat{B}(\mu_k(T))_{ji}$,
we need to know left and right $\add \Tbs$-approximations for $T_i$, where 
\[
  \Tbs=(T/(T_k\oplus T_i)) \oplus T^*_k.
\]
But any map $f\colon T_i\to T_k^*$ factors 
through the right approximation $r_k\colon T'_k\to T_k^*$ in \eqref{eq:right-mutation}.
\[
\begin{tikzpicture}[scale=1.6]
\draw (0,1) node (A2) {$T_i$};
\draw (0,0) node (B2) {$T_k^*$};
\draw (-1,0) node (B1) {$T'_k$};
\draw [cdmap] (A2) to node [right] {\maplab $f$} (B2);
\draw [cdmap] (B1) to node [below] {\maplab $r_k$} (B2);
\draw [cdmap, densely dotted] (A2) to node [above left] {\maplab $s$} (B1);
\end{tikzpicture}
\]
As $T'_k$ has no summand isomorphic to $T_i$ or $T_k$, 
the map $s$ factors through $a$ and so $f$ factors through $a$. 
Thus $a$ is actually a left $\Tbs$-approximation.
Note that this implies that $\Ext^1_A(X,T^*_k)=0$ and hence $\Ext_A^{1}(T_k^*, X)=0$,
by the 2-Calabi-Yau property of $\CM(A)$.

Now take the exchange sequence \eqref{eq:left-mutation} for $T_i$ instead of $T_k$,
together with its push-out along the map $g\colon T_i^*\to X$, 
to obtain the following commutative diagram of four short exact sequences. 
\[
\begin{tikzpicture}[xscale=2, yscale=1.6]
\pgfmathsetmacro{\hoff}{0.7}
\pgfmathsetmacro{\voff}{0.7}
\draw (0,1) node (A1) {$T^*_i$};
\draw (1,1) node (A2) {$T''_i$};
\draw (2,1) node (A3) {$T_i$};
\draw (0,0) node (B1) {$X$};
\draw (1,0) node (B2) {$Y$};
\draw (2,0) node (B3) {$T_i$};
\draw (0,-1) node (C1) {$(T_k^*)^\om$};
\draw (1,-1) node (C2) {$(T_k^*)^\om$};
\draw (A1)++(-\hoff,0) node (a1) {$\cdzero$};
\draw (A3)++(\hoff,0) node (a3) {$\cdzero$};
\draw (B1)++(-\hoff,0) node (b1) {$\cdzero$};
\draw (B3)++(\hoff,0) node (b3) {$\cdzero$};
\draw (A1)++(0,\voff) node (aa1) {$0$};
\draw (A2)++(0,\voff) node (aa2) {$0$};
\draw (C1)++(0,-\voff) node (cc1) {$0$};
\draw (C2)++(0,-\voff) node (cc2) {$0$};
\foreach \P/\Q in {a1/A1, A1/A2, A2/A3, A3/a3, b1/B1, B1/B2, B3/b3,
  B1/C1, A2/B2, B2/C2, aa1/A1, aa2/A2, C1/cc1, C2/cc2}
  \draw [cdmap] (\P) to (\Q);
\draw [cdmap] (A1) to node [left] {\maplab $g$} (B1);
\draw [cdmap] (B2) to node [above] {\maplab $b$} (B3);
\foreach \P/\Q in {A3/B3, C1/C2}
  \draw [equals] (\P) to (\Q);
\end{tikzpicture}
\] 
By assumption, there is no arrow $k$ to $i$, 
so there is no summand of $T''_i$ isomorphic to $T_k$. 
Hence $\Ext_A^1(T_k^*, T''_i)=0$ and so $Y=T''_i\oplus (T_k^*)^{\om}$.
Furthermore $b$ is a right $\Tbs$-approximation, 
because $b$ has a component that is the 
right approximation $T''_i\to T_i$ and $\Ext_A^{1}(T_k^*, X)=0$.

Denote the multiplicity of a module $M$ in $N$ by $\mult{M}{N}$. 
Because $(\mat{B}(\mu_k(T)))_{ji}$ is the difference of  the number of arrows from $j$ to $i$ 
and the number arrows from $i$ to $j$ in the Gabriel quiver of $\End_A(\mu_k(T))$, we have
\begin{align*}
\mat{B}(\mu_k(T))_{ji}
&= \abs{\{\text{arrows } j\to i\}} - \abs{\{\text{arrows } i\to j\}} \\
&= \mult{T_j}{T''_i\oplus (T_k^*)^\om} - \mult{T_j}{D_i'\oplus (T'_k)^\om} \\
&= \mult{T_j}{T''_i} - \mult{T_j}{T_i'} - \om\mult{T_j}{T'_k}
\quad\text{(as $T_j$ isn't $T_k^*$ or $T_k$)} \\
&= \brakp{b_{ji}} - \brakp{-b_{ji}} -\om \brakp{-b_{jk}} \\
&= b_{ji} -\om \brakp{-b_{jk}} \\
&= \mu_{k}(\mat{B}(T))_{ji},
\end{align*} 
where $\brakp{x}=\max\{0, x\}$ (as in \S\ref{sec:qca}).
This completes the proof when $i,j\neq k$ and $b_{ik}\geq 0$.
The case where $b_{ik}\leq 0$, i.e.~there is no arrow from $i$ to $k$, is similar 
(or rather dual)
and we omit the proof.
\end{proof}

\begin{remark}
The maps $a$ and $b$, in the proof of Theorem~\ref{thm:mutationmatrix}, 
are not necessarily minimal approximations. 
However, by the nature of the two short exact sequences containing $a$ and $b$, the middle terms $Y$ 
and $D_i'\oplus (T'_k)^\om$ have the same number of extra summands that are isomorphic to $T_j$.
\end{remark}

Recalling the invariant $\kappa(M, N)$ from Definition~\ref{def:kappa}, we define
\begin{equation}\label{eq:lam-def}
 \lambda(M, N) = \kappa(N,M) - \kappa(M, N).
\end{equation} 
Given a cluster tilting object $T=\oplus_{i=1}^{\cn} T_i$, we define an ${\cn\times \cn}$-matrix 
$\mat{L}=\mat{L}(T)=(\lambda_{ij})$ associated to $T$, and the vertex $v$, by 
\begin{equation}\label{eq:LfromT}
\lambda_{ij}=\lambda(T_i, T_j).
\end{equation} 

It turns out that this is an appropriate generalisation of the construction of 
Geiss-Leclerc-Schr\"oer~\cite[\S10.2]{GLS13} 
to the current context, when $\Hom_A(M, N)$ is infinite dimensional, 
so their definition cannot be used.
Indeed, their proofs of \cite[Props 10.1 \& 10.2]{GLS13} can be adapted to give a similar result, as follows.

\begin{theorem}\label{thm:main}
The two matrices $\mat{B}$ and $\mat{L}$ associated to $T$, as above, are compatible. 
Furthermore, mutation of cluster tilting objects is consistent with mutation of seed data,
that is, the pair associated to the mutated object $\mu_k(T)$
is the mutated pair $\mu_k(\mat{B}, \mat{L})$. 
\end{theorem}

\begin{proof}
Applying $\Hom_A(T_\ell, -)$ and $\Hom_Z(e_0T_\ell, e_0-)$ to \eqref{eq:right-mutation},
and using a degenerate case of the Snake Lemma, we obtain
\[
\begin{tikzpicture}[xscale=4.2, yscale=1.6]
\pgfmathsetmacro{\hoff}{0.6}
\pgfmathsetmacro{\voff}{0.7}
\draw (1,1) node (A1) {$\Hom_A(T_\ell, T_k)$};
\draw (2,1) node (A2) {$\Hom_A(T_\ell, T'_k)$};
\draw (3,1) node (A3) {$\Hom_A(T_\ell, T_k^*)$};
\draw (1,0) node (B1) {$\Hom_Z(e_0T_\ell, e_0T_k)$};
\draw (2,0) node (B2) {$\Hom_Z(e_0T_\ell, e_0T'_k)$};
\draw (3,0) node (B3) {$\Hom_Z(e_0T_\ell, e_0T_k^*)$};
\draw (1,-1) node (C1) {$K(T_\ell, T_k)$};
\draw (2,-1) node (C2) {$K(T_\ell, T'_k)$};
\draw (3,-1) node (C3) {$K(T_\ell , T^*_k)$};
\draw (A1)++(-\hoff,0) node (a1) {$\cdzero$};
\draw (A3)++(\hoff,0) node (a3) {$\cdzero$};
\draw (B1)++(-\hoff,0) node (b1) {$\cdzero$};
\draw (B3)++(\hoff,0) node (b3) {$\cdzero$};
\draw (C1)++(-\hoff,0) node (c1) {$\cdzero$};
\draw (C3)++(\hoff,0) node (c3) {$\cdzero$};
\draw (A1)++(0,\voff) node (aa1) {$0$};
\draw (A2)++(0,\voff) node (aa2) {$0$};
\draw (A3)++(0,\voff) node (aa3) {$0$};
\draw (C1)++(0,-\voff) node (cc1) {$0$};
\draw (C2)++(0,-\voff) node (cc2) {$0$};
\draw (C3)++(0,-\voff) node (cc3) {$0$};
\foreach \T/\H in {a1/A1, A1/A2, A2/A3, A3/a3, b1/B1, B1/B2, B2/B3, B3/b3, c1/C1, C1/C2, C2/C3, C3/c3,
  A1/B1, A2/B2, A3/B3, B1/C1, B2/C2, B3/C3, aa1/A1, aa2/A2, aa3/A3, C1/cc1, C2/cc2, C3/cc3}
  \draw [cdmap] (\T) to (\H);
\end{tikzpicture}
\] 
where the vertical short exact sequences are as in \eqref{eq:def-Kv}.
Note that we have used the fact that $\Ext_A^1(T_\ell,T_k)=0$ for all $\ell$.
Thus 
\begin{equation}\label{eq:kapTT1}
  \kappa(T_\ell, T'_k) =  \kappa(T_\ell, T_k) + \kappa(T_\ell, T^*_k).
\end{equation}
Similarly, applying $\Hom_A(-,T_\ell)$ to \eqref{eq:left-mutation}, we obtain 
\begin{equation}\label{eq:kapT2T}
   \kappa(T''_k,T_\ell) = \kappa(T_k,T_\ell) + \kappa(T^*_k,T_\ell) .
\end{equation}
In the case $\ell\neq k$, we also have $\Ext_A^1(T_\ell, T_k^*) =0= \Ext_A^1(T_k^*,T_\ell)$ 
and can equally apply $\Hom_A(T_\ell, -)$ to \eqref{eq:left-mutation} 
and $\Hom_A(-,T_\ell)$ to \eqref{eq:right-mutation} to obtain
\begin{align}
 \kappa(T_\ell, T''_k) &= \kappa(T_\ell, T_k) + \kappa(T_\ell, T^*_k), \label{eq:kapTT2} \\
 \kappa(T'_k,T_\ell) &= \kappa(T_k,T_\ell) + \kappa(T^*_k,T_\ell). \label{eq:kapT1T}
\end{align}
Combining \eqref{eq:kapTT1}--\eqref{eq:kapT1T} 
gives $\kappa(T_\ell, T'_k)= \kappa(T_\ell, T''_k)$ 
and $\kappa(T'_k,T_\ell)= \kappa(T''_k,T_\ell)$, 
and hence $\lambda(T'_k,T_\ell)= \lambda(T''_k,T_\ell)$, that is,
\[
 \sum_{j:b_{jk}<0} (-b_{jk}) \lambda(T_j, T_\ell) 
  = \sum_{j:b_{jk}>0} b_{jk} \lambda(T_j,T_\ell).
\]
In other words,
\[ 
  \sum_{j} b_{jk} \lambda_{j\ell} = 0,
\]
which is the $\ell\neq k$ part of the condition \eqref{eq:compat} that $\mat{B}$ and $\mat{L}$ are compatible. 

In the case $\ell=k$, following \cite[Prop 8.1]{GLS08} and \cite[Cor 4.6]{JKS}, we have 
\[\Ext_A^1(T_k, T_k^*) \isom \CC \isom \Ext_A^1(T_k^*,T_k).\]
So, when we apply $\Hom_A(-,T_k)$ and $\Hom_Z(-,e_0T_k)$ to \eqref{eq:right-mutation}, 
we obtain the following commutative diagram with exact rows and columns
\[
\begin{tikzpicture}[xscale=4.0, yscale=1.6]
\pgfmathsetmacro{\hoff}{0.6}
\pgfmathsetmacro{\voff}{0.7}
\draw (1,1) node (A1) {$\Hom_A(T_k^*, T_k)$};
\draw (2,1) node (A2) {$\Hom_A(T'_k, T_k)$};
\draw (3,1) node (A3) {$\Hom_A(T_k, T_k)$};
\draw (1,0) node (B1) {$\Hom_Z(e_0T_k^*, e_0T_k)$};
\draw (2,0) node (B2) {$\Hom_Z(e_0T'_k, e_0T_k)$};
\draw (3,0) node (B3) {$\Hom_Z(e_0T_k, e_0T_k)$};
\draw (1,-1) node (C1) {$\ker\psi$};
\draw (2,-1) node (C2) {$K(T'_k, T_k)$};
\draw (3,-1) node (C3) {$K(T_k, T_k)$};
\draw (A1)++(-\hoff,0) node (a1) {$\cdzero$};
\draw (A3)++(\hoff,0) node (a3) {$\CC_{\ph}$};
\draw (a3)++(0.25,0) node (ay) {$\cdzero$};
\draw (B1)++(-\hoff,0) node (b1) {$\cdzero$};
\draw (B3)++(\hoff,0) node (b3) {$\cdzero$};
\draw (C1)++(-\hoff,0) node (c1) {$\cdzero$};
\draw (A1)++(0,\voff) node (aa1) {$0$};
\draw (A2)++(0,\voff) node (aa2) {$0$};
\draw (A3)++(0,\voff) node (aa3) {$0$};
\draw (C2)++(0,-\voff) node (cc2) {$0$};
\draw (C3)++(0,-\voff) node (cc3) {$0$};
\foreach \T/\H in {a1/A1, A1/A2, A2/A3, A3/a3, a3/ay, b1/B1, B1/B2, B2/B3, B3/b3, c1/C1, C1/C2,
  A1/B1, A2/B2, A3/B3, B1/C1, B2/C2, B3/C3, aa1/A1, aa2/A2, aa3/A3, C2/cc2, C3/cc3}
  \draw [cdmap] (\T) to (\H);
\draw [cdmap] (C2) to node [above] {\maplab $\psi$} (C3);
\end{tikzpicture}
\] 
Applying the Snake Lemma to the middle two columns shows that $\psi$ is surjective and that there is a short exact sequence
\[
0 \to K(T_k^*, T_k) \to \ker\psi \to \CC \to 0.
\]
Hence
\begin{equation}\label{eq:kapT1Tk}
  \kappa(T'_k,T_k) = \kappa(T_k,T_k) + \kappa(T^*_k,T_k) +  1 .
\end{equation}
Similarly, applying $\Hom(T_k,-)$ to \eqref{eq:left-mutation}, we obtain 
\begin{equation}\label{eq:kapTkT2}
  \kappa(T_k, T''_k) = \kappa(T_k, T_k) + \kappa(T_k, T^*_k) +  1 .
\end{equation}
Using \eqref{eq:kapTT1} and \eqref{eq:kapT2T}, for $\ell=k$, 
together with \eqref{eq:kapT1Tk} and \eqref{eq:kapTkT2}, we obtain 
\begin{equation}\label{eq:compat}
  \sum_{j} b_{jk} \lambda_{jk} =  \lambda(T''_k,T_k) - \lambda(T'_k,T_k) = 2 ,
\end{equation}
which is the $\ell=k$ part of the compatibility condition \eqref{eq:compat}.

The equations \eqref{eq:kapTT1}--\eqref{eq:kapT1T} also show that, when $\ell\neq k$,
\begin{align*}
\label{eq:newlam}
 \lambda(T_k^*,T_\ell)+\lambda(T_k,T_\ell)
 &= \lambda(T'_k,T_\ell) = \sum_{j:b_{jk}<0} (-b_{jk}) \lambda(T_j,T_\ell) \\
 &= \lambda(T''_k,T_\ell) = \sum_{j:b_{jk}>0} b_{jk} \lambda(T_j,T_\ell).
\end{align*}
Comparing this with Remark~\ref{rem:compat}, 
we see that the mutated matrix $\mu_k(\mat{L})$ is defined by the function $\lambda$ applied to $\mu_k(T)$.
Together with Theorem \ref{thm:mutationmatrix}, this 
completes the proof that mutation of cluster tilting objects is consistent with mutation of seed data.
\end{proof}

We now show that, when $T=T_\maxNC$ as in \eqref{eq:TmaxNC}, our matrix $\mat{L}$ computes the quasi-commutation rules,
as in Theorem \ref{thm:LZ}, for the corresponding quantum minors.

\begin{lemma}\label{lem:kappa}
Suppose that $I$ and $J$ are non-crossing and that $J\setminus {I}=J'\cup J''$ so that $J'< (I\setminus J) <J''$,
i.e.~case (i) of Definition~\ref{def:non-crossing}. 
Then 
\[
  \kappa(M_I, M_J)=|J'| \quad\text{and}\quad \kappa(M_J, M_I)=|J''|.
\]
\end{lemma}

\begin{proof}
By Remark 5.4 in \cite{JKS}, $\Hom_A(M_I, M_J)$ is generated by $t^{\alpha}$, 
where $\alpha\in \NN^{C_0}$ is the minimal exponent vector satisfying
\[
\alpha_{ha}-\alpha_{ta} = 
\begin{cases} 
   -1 & \text{if $a\in J\setminus I$,}\\
  +1 & \text{if $a\in I\setminus J$,}\\
    0 & \text{otherwise.}
\end{cases}
\]
Here $ha$ and $ta$ are the head and tail, respectively, of the arrow $x_a$. 
Thus $\alpha$ decreases by 1 on each edge in $J'$, 
then increases by 1 on each edge in $I\setminus J$, 
then decreases by 1 on each edge in $J''$.
To be minimal, $\alpha$ must be zero somewhere, 
which must be on the vertices between  $J'$ and $I\setminus J$.
Hence $\kappa(M_I, M_J)=\alpha_0=|J'|$.

Similarly, $\Hom_A(M_J, M_I)$ is generated by $t^{\beta}$, 
where $\beta$ increases by 1 on each edge in $J'$, then decreases by 1 on each edge in $I\setminus J$, then
increases by 1 on each edge in $J''$. 
Furthermore, $\beta$ must be zero on the vertices between  $I\setminus J$ and $J''$
and so $\kappa(M_J, M_I)=\beta_0=|J''|$.
\end{proof}

\begin{theorem} \label{thm:quasicomm}
Let $M_I$, $M_J$ be rank-1 modules with $I$ and $J$ non-crossing sets.
Then 
\[
  c(I,J) = \lambda(M_I, M_J) = \kappa(M_J, M_I)-\kappa(M_I, M_J)
\]
where $c(I,J)$, as in Definition~\ref{def:non-crossing},
is the exponent in the quasi-commutation rule of Theorem~\ref{thm:LZ}:
\[
 \qminor{I}\qminor{J}=q^{c(I,J)}\qminor{J}\qminor{I}.
\]
\end{theorem}

\begin{proof}
Swapping the roles of $I$ and $J$ if necessary, 
we may assume that case (i) of Definition~\ref{def:non-crossing} holds 
and then this is an immediate application of Lemma~\ref{lem:kappa}.
\end{proof}

\section{A link with mirror symmetry}\label{sec:maxdiag}

There is a correspondence between (isomorphism classes of) 
rank one $A$-modules $M_J$ and Young diagrams $\lambda_J$ contained
in a rectangle of width $m$ and height $n-m$.
The labelling of the Young diagram $\lambda_J$
is by the west steps in the NE-SW walk that bounds the diagram,
as in \cite[\S2.3]{RW17}.
The labelling of the module $M_J$
is as in Section~\ref{sec:gcc}, that is, as in \cite[\S5]{JKS}.
Alternatively, by rotating the Young diagram into Russian orientation, i.e.~by 135 degrees,
we effectively obtain a picture of the module $\pi M_J$, where $\pi$ is as in \eqref{eq:pi-def-old}.
Put another way, the NE-SW walk
becomes the `contour' (cf.~\cite[\S6]{JKS}) of the module $M_J$
(see Figure~\ref{fig:russian} for an illustration).

\begin{figure}[h] 
\begin{tikzpicture} [scale=0.4,
upbdry/.style={thick, gray},
lowbdry/.style={thick, gray},
boxes/.style={thick, blue},
ridge/.style={very thick, red},
>={Stealth[inset=2.5pt,length=4.5pt,angle'=40,round]},
outarr/.style={->, gray},
midarr/.style={->,blue},
rdgarr/.style={->,red},
keyarr/.style={->,black},
outdot/.style={gray, fill= gray},
middot/.style={blue, fill=blue}]

\newcommand{\abit}{0.133}

\begin{scope} 
\draw [boxes] (0,-1)--(3,-1) (0,-2)--(1,-2) (1,0)--(1,-2) (2,0)--(2,-2);
\draw [upbdry] (0,-5)--(4,-5)--(4,0);
\draw [lowbdry] (0,-5)--(0,0)--(4,0);
\draw [ridge] (4,0)--(3,0)--++(0,-2)--++(-2,0)--++(0,-1)--++(-1,0)--(0,-5);
\draw (3.5,0) node [below=-2pt] {\tiny 1};
\draw (2.5,-2) node [below=-2pt] {\tiny 4};
\draw (1.5,-2) node [below=-2pt] {\tiny 5};
\draw (0.5,-3) node [below=-2pt] {\tiny 7};
\end{scope}

\begin{scope} [shift={(10,-5)},rotate=135]
\draw [boxes] (0,-1)--(3,-1) (0,-2)--(1,-2) (1,0)--(1,-2) (2,0)--(2,-2);
\draw [upbdry] (0,-5)--(4,-5)--(4,0);
\draw [lowbdry] (0,-5)--(0,0)--(4,0);
\draw [ridge] (4,0)--(3,0)--++(0,-2)--++(-2,0)--++(0,-1)--++(-1,0)--(0,-5);
\draw (3.5,0) node [above right=-3pt] {\tiny 1};
\draw (2.5,-2) node [above right=-3pt] {\tiny 4};
\draw (1.5,-2) node [above right=-3pt] {\tiny 5};
\draw (0.5,-3) node [above right=-3pt] {\tiny 7};
\end{scope}

\begin{scope} [shift={(20,-4)}, scale=0.7]

\draw [dashed] (-4,5)--(-4,-2) (5,-1)--(5,6);
\draw [dotted] (-4,-2)--++(1,1)--++(1,-1)--++(1,1)--++(1,-1)--++(1,1)--++(1,-1)--++(1,1)--++(1,-1)--++(1,1);

\foreach \x/\y in {4/0, 3/2, 2/2, 1/3}
 { \draw [rdgarr] (-\x+\y+\abit,\x+\y-\abit)--++(1-2*\abit,-1+2*\abit);
  \draw [thick,red] (-\x+\y+\abit,\x+\y-\abit)--++(1-3*\abit,-1+3*\abit);}
\foreach \x/\y in {3/1, 2/1, 1/2, 1/1}
  \draw [midarr] (-\x+\y+\abit,\x+\y-\abit)--++(1-2*\abit,-1+2*\abit);
\foreach \x/\y in {3/0, 3/-1, 2/0, 2/-1, 2/-2, 1/0, 1/-1, 0/4, 0/3, 0/2, 0/1, 0/0, -1/3, -1/2, -1/1, -2/2}
  \draw [outarr] (-\x+\y+\abit,\x+\y-\abit)--++(1-2*\abit,-1+2*\abit);
    
\foreach \x/\y in {3/1, 3/2, 1/3, 0/4, 0/5}
  {\draw [rdgarr] (-\x+\y-\abit,\x+\y-\abit)--++(-1+2*\abit,-1+2*\abit);
  \draw [thick,red] (-\x+\y-\abit,\x+\y-\abit)--++(-1+3*\abit,-1+3*\abit);}
\foreach \x/\y in {2/2, 2/1, 2/1, 1/2, 1/1}
  \draw [midarr] (-\x+\y-\abit,\x+\y-\abit)--++(-1+2*\abit,-1+2*\abit);
\foreach \x/\y in {3/0, 2/0, 2/-1, 1/0, 1/-1, 0/3, 0/2, 0/1, 0/0, -1/4, -1/3, -1/2, -1/1, -2/3, -2/2}
  \draw [outarr] (-\x+\y-\abit,\x+\y-\abit)--++(-1+2*\abit,-1+2*\abit);
  
\foreach \x/\y in {3/1, 3/2, 2/2, 2/1, 1/3, 1/2, 1/1}
  \draw [middot] (-\x+\y,\x+\y) circle (3pt);
\foreach \x/\y in {4/0, 3/0, 3/-1, 2/0, 2/-1, 2/-2, 1/0, 1/-1, 0/2, 0/1, 0/0, 0/3, 0/4, 0/5, -1/4, -1/3, -1/2, -1/1, -2/3, -2/2}
  \draw [outdot] (-\x+\y,\x+\y) circle (3pt);

\draw [keyarr] (7-\abit,4-\abit)--++(1-2*\abit,-1+2*\abit);
\draw [keyarr] (8-\abit,2-\abit)--++(-1+2*\abit,-1+2*\abit);
\draw (7.8,3.8) node {\tiny $x$};
\draw (7.8,1) node {\tiny $y$};
\end{scope}

\end{tikzpicture}
\caption{Young diagram $\lambda_{1457}$ (in English and Russian orientation)
and module $M_{1457}$, in the case $(\mm,\nn)=(4,9)$.}
\label{fig:russian}
\end{figure}

Recall from \cite[Def.~14.3]{RW17} the quantity 
\begin{equation}
 \mxd{I}{J}=\maxdiag(\partn{I} \setminus \partn{J}),
\end{equation}
which measures the maximum width (diagonally in English orientation or vertically in Russian orientation)
of the set-theoretic difference of the Young diagrams $\partn{I}$ and $\partn{J}$. 
For example, Figure~\ref{fig:maxdiag} shows that $\mxd{2389}{1457}=2$ 
(and also that $\mxd{1457}{2389})=1$).

\begin{lemma} \label{lem:kap=md}
For any $m$-subsets $I,J$, we have $\kappa(M_I,M_J) = \mxd{I}{J}$. 
\end{lemma}
\begin{proof}
Reformulating the definition, $\mxd{I}{J}$ is the maximum height of the contour of $M_I$ 
above the contour of $M_J$, when they are made to line up at the vertex~0.
In other words, it is the minimum distance we need 
to lower the contour of $M_I$ to be below the contour of $M_J$ (see Figure~\ref{fig:maxdiag} for an illustration).

On the other hand, recall from Definitions~\ref{def:Kv} and~\ref{def:kappa} that $\kappa(M_I,M_J)$ is the 
codimension of $\Hom_A(M_I,M_J)$ in $\Hom_Z(e_0 M_I,e_0 M_J)$, where the inclusion is via the restriction functor
$\funE\colon M\mapsto e_0M$, as in \eqref{eq:E}.
In this case, both Hom spaces are free rank one $Z$-modules.

Now the right-hand picture in Figure~\ref{fig:maxdiag} represents 
the embedding $\eta\colon M_I \hookrightarrow M_J$ that generates $\Hom_A(M_I,M_J)$,
while the height difference $h$ at the vertex 0 represents the fact that $\funE(\eta)=t^h \nu$,
where $\nu$ is the generator of $\Hom_Z(e_0 M_I,e_0 M_J)$.
Thus $h$ is the required codimension.
\end{proof}

As an extra observation, note that the left-hand picture in Figure~\ref{fig:maxdiag} represents 
embeddings $P_0\hookrightarrow M_I \hookrightarrow P_{\nn-\mm}$ and the same for $M_J$.
Two of these embeddings are the generators of $\Hom_A(M_I,P_{\nn-\mm})$ and $\Hom_A(P_0,M_J)$,
which are identified with~$\nu$ under the isomorphisms in Lemma ~\ref{lem:phi-equivs},
since $P_{\nn-\mm}=\funJ e_0 M_J$ and $P_0=\funP e_0 M_I$.
The other two embeddings are the natural maps $\alpha$ and $\beta$ of Lemma~\ref{lem:pi-om}.
This provides alternative ways to conclude the proof of Lemma~\ref{lem:kap=md}.

\begin{figure}[h]
\begin{tikzpicture}[scale=0.3,
  htarrow/.style={latex-latex, thick}]
  
\begin{scope}
\draw [dotted] (2,-2)--++(1,1) (3,-3)--++(5,5);
\foreach \j in {1,...,4}
  \draw [dotted] (2,-2)++(\j,\j)--++(2,-2);

\draw [gray, thick] (0,0)--++(4,-4)--++(5,5);
\draw [gray, thick] (0,0)--++(5,5)--++(4,-4);

\draw [teal, very thick] (0,0)
 --++(1,1)--++(1,-1)--++(1,-1)--++(1,1)--++(1,1)--++(1,1)--++(1,1)--++(1,-1)--++(1,-1);
\draw [red, very thick] (0,0)
 --++(1,-1)--++(1,1)--++(1,1)--++(1,-1)--++(1,-1)--++(1,1)--++(1,-1)--++(1,1)--++(1,1);

\draw [htarrow] (7,2.9)--++(0,-3.8);
\end{scope}

\begin{scope} [shift={(15,1)}]
\draw [dotted] (0,-2)--++(1,1) (1,-3)--++(3,3) (1,-5)--++(4,4);
\draw [dotted] (0,-4)--++(1,-1) (0,-2)--++(1,-1) (5,-1)--++(4,-4) (8,0)--++(1,-1);
\draw [dotted] (1,-1)--++(4,-4)--++(4,4) (2,0)--++(5,-5)--++(2,2);

\draw [teal, very thick] (0,-4)
 --++(1,1)--++(1,-1)--++(1,-1)--++(1,1)--++(1,1)--++(1,1)--++(1,1)--++(1,-1)--++(1,-1);
\draw [red, very thick] (0,0)
 --++(1,-1)--++(1,1)--++(1,1)--++(1,-1)--++(1,-1)--++(1,1)--++(1,-1)--++(1,1)--++(1,1);

\draw [htarrow] (0,-0.1)--++(0,-3.8);
\draw [htarrow] (9,1-0.1)--++(0,-3.8);
\end{scope}

\end{tikzpicture}
\caption{Demonstration that 
$\mxd{\textcolor{teal}{2389}}{\textcolor{red}{1457}} = 2 =
\kappa(M_{\textcolor{teal}{2389}},M_{\textcolor{red}{1457}})$.}
\label{fig:maxdiag}
\end{figure}

\begin{remark}
Combining Theorem~\ref{thm:quasicomm} and Lemma~\ref{lem:kap=md},
one may deduce that, if $I$ and $J$ non-crossing sets, then 
\begin{equation}
 q^{\mxd{I}{J}} \qminor{I} \qminor{J} = q^{\mxd{J}{I}} \qminor{J} \qminor{I}.
\end{equation}
Indeed, it is a simple combinatorial exercise to check this directly. 
In other words, in case (i) of Definition~\ref{def:non-crossing},
\[
\mxd{I}{J} = |J'|, \qquad
\mxd{J}{I} = |J''|,
\]
and a similar formula holds in case (ii).
This seems to be a new observation and an intriguing one,
given the known role that $\mxd{I}{J}$ plays in~\cite{RW17} and~\cite{FW}.
\end{remark}

\begin{remark}\label{rem:cluschar}
From its description as the cokernel of $\beta_*(M)$ in Lemma \ref{lem:phi-equivs}, 
we can see that $K(T, N)$ is an $\End_A(T)$-module,
for any $T\in\CM(A)$.
In particular, for $T_\maxNC$, as in \eqref{eq:TmaxNC}, 
$K(T_\maxNC, N)$ is a representation of the quiver
dual to a certain reduced plabic graph $G$ with faces labelled by the maximal noncrossing set $\maxNC$ 
(see \cite{BKM}, or \cite[Defs.~3.5 \& 3.8]{RW17} for the closely related quiver with boundary arrows omitted).
The dimension vector of this representation is given by 
\[ 
  \dimvec K(T_\maxNC, N) = (\kappa(M_J, N):J\in\maxNC).
\]
Thus we can interpret \cite[Cor~16.19]{RW17} together with Lemma~\ref{lem:kap=md} as saying that
\begin{equation}\label{eq:dim-val-min}
  \valG(\minor{I}) = \dimvec K(T_\maxNC, M_I),
\end{equation}
where $\valG$ is defined in \cite[Def~8.1]{RW17}
and $\minor{I}$ is the Pl\"ucker coordinate with label~$I$.
Note that the projective $P_0=M_{1\cdots m}$
at the vertex $0$ is always a summand of $T_\maxNC$ 
and $K(P_0, N)=0$ for any $N$. 
By convention, this trivial component is omitted from the definition of $\valG$.

In light of \eqref{eq:dim-val-min} and the fact that $\minor{I}$ is the cluster character of $M_I$,
it is natural to conjecture that, for any $N\in \CM(A)$,
 \begin{equation}\label{eq:dim-val-conj}
  \valG(\cluschar{N}) = \dimvec K(T_\maxNC, N),
\end{equation}
where $\cluschar{N}$ is the cluster character of $N$.

\end{remark}

\section{The Grassmannian cluster algebra} \label{sec:grCA}

To understand the quantum case, we start by recalling in some detail how the 
Grassmannian homogeneous coordinate ring $\CC[\Grr]$ becomes a (graded) cluster algebra.
We follow the original work of Scott~\cite{Sc06}, enhanced by the combinatorial fact,
now proved by Oh-Postnikov-Speyer~\cite[Thm 1.6]{OPS},
that every maximal non-crossing set $\maxNC$ (Def~\ref{def:non-crossing})
corresponds (one-to-one) to a certain type of Postnikov alternating strand diagram, or a 
reduced plabic graph $G_\maxNC$ or, dually, a quiver $Q_\maxNC$ and thus an exchange matrix $\mat{B}_\maxNC$.
The set $\maxNC$ canonically labels the faces of $G_\maxNC$, and hence the vertices of $Q_\maxNC$ and the rows and columns of $\mat{B}_\maxNC$ (see \cite[\S5]{Sc06}  or  \cite{BKM}, \cite{OPS} for details).

Then $X_\maxNC=\{X_I:I\in \maxNC\}$, together with the exchange matrix $\mat{B}_\maxNC$, is an initial seed for a cluster algebra $C(X_\maxNC,\mat{B}_\maxNC)$, which can be graded by giving all the $X_I$ degree 1, since the quiver $Q_\maxNC$ is balanced, in the sense of \cite[Lemma~2.1]{JKS}.

Now,  by \cite[Thm 1.4]{OPS}, 
any two such maximal non-crossing sets can be related by a sequence of mutations or {\it geometric exchanges}.
More precisely, suppose that $a,b,c,d$ are cyclically ordered indices and a maximal non-crossing set $\maxNC$ 
contains the minor labels $Jab$, $Jbc$, $Jcd$, $Jad$ and $Jac$, where 
$J$ is disjoint from $\{a,b,c,d\}$ and $Jab$ is short-hand for $J\cup\{a,b\}$, etc. 
Then there will be another maximal non-crossing set $\maxNC'$ in which $Jac$ is replaced by $Jbd$.
The local change in the corresponding alternating strand diagram and 
quiver is illustrated in Figure~\ref{fig:geomexch}.
Thus, the quiver $Q_{\maxNC'}$ is obtained from $Q_{\maxNC}$ by quiver mutation.

\begin{figure}
\begin{tikzpicture}[scale=0.6,
  strand/.style={teal,dashed,thick},
  quivarrow/.style={red, -latex, thick},
  doublearrow/.style={black, latex-latex, very thick}]
\newcommand{\strarrow}{\arrow{angle 60}}
\newcommand{\dotrad}{0.1cm} 
\newcommand{\bdrydotrad}{{0.8*\dotrad}} 

\draw [strand] plot coordinates {(4,-2) (1.8,-0.2) (0.8,0.8) (-0.2, 1.8) (-2,4)}[postaction=decorate,decoration={markings,
mark= at position 0.17 with \strarrow,
mark= at position 0.5 with \strarrow,
mark= at position 0.83 with \strarrow}];

\draw [strand] plot coordinates {(-4,2) (-1.8,0.2) (-0.8,-0.8) (0.2,-1.8) (2,-4)}[postaction=decorate,decoration={markings,
mark= at position 0.17 with \strarrow,
mark= at position 0.5 with \strarrow,
mark= at position 0.83 with \strarrow}];

\draw [strand] plot coordinates {(-4,-2) (-1.8,-0.2) (-0.8,0.8) (0.2,1.8) (2,4)}[postaction=decorate,decoration={markings,
mark= at position 0.17 with \strarrow,
mark= at position 0.5 with \strarrow,
mark= at position 0.83 with \strarrow}];

\draw [strand] plot coordinates {(4,2) (1.8,0.2) (0.8,-0.8) (-0.2,-1.8) (-2,-4)}[postaction=decorate,decoration={markings,
mark= at position 0.17 with \strarrow,
mark= at position 0.5 with \strarrow,
mark= at position 0.83 with \strarrow}];

\draw (4.3,-2.3) node {$b$};
\draw (-4.3,-2.3) node {$c$};
\draw (-4.3,2.3) node {$d$};
\draw (4.3,2.3) node {$a$};

\draw (0,0) node (ac) {$Jac$};
\draw (0,3.5) node (ab) {$Jab$};
\draw (3.5,0) node (bc) {$Jbc$};
\draw (0,-3.5) node (cd) {$Jcd$};
\draw (-3.5,0) node (ad) {$Jad$};

\draw [quivarrow] (ad)--(ac);
\draw [quivarrow] (bc)--(ac); 
\draw [quivarrow] (ac)--(cd); 
\draw [quivarrow] (ac)--(ab);

\begin{scope}[shift={(12,0)}]

\draw [strand] plot[smooth] coordinates {(4,-2) (2,-1.8) (0,-1.3) (-1.3,0) (-1.8,2) (-2,4)}[postaction=decorate,decoration={markings,
mark= at position 0.1 with \strarrow,
mark= at position 0.3 with \strarrow,
mark= at position 0.5 with \strarrow,
mark= at position 0.74 with \strarrow,
mark= at position 0.92 with \strarrow
}];

\draw [strand] plot[smooth] coordinates {(-4,2) (-2,1.8) (0,1.3) (1.3,0) (1.8,-2) (2,-4)}[postaction=decorate,decoration={markings,
mark= at position 0.1 with \strarrow,
mark= at position 0.3 with \strarrow,
mark= at position 0.5 with \strarrow,
mark= at position 0.74 with \strarrow,
mark= at position 0.92 with \strarrow}];

\draw [strand] plot[smooth] coordinates {(-4,-2) (-2,-1.8) (0,-1.3) (1.3,0) (1.8,2) (2,4)}[postaction=decorate,decoration={markings,
mark= at position 0.1 with \strarrow,
mark= at position 0.3 with \strarrow,
mark= at position 0.5 with \strarrow,
mark= at position 0.74 with \strarrow,
mark= at position 0.92 with \strarrow}];

\draw [strand] plot[smooth] coordinates {(4,2) (2,1.8) (0,1.3) (-1.3,0) (-1.8,-2) (-2,-4)}[postaction=decorate,decoration={markings,
mark= at position 0.1 with \strarrow,
mark= at position 0.3 with \strarrow,
mark= at position 0.5 with \strarrow,
mark= at position 0.74 with \strarrow,
mark= at position 0.92 with \strarrow}];

\draw (4.4,-2.1) node {$b$};
\draw (-4.4,-2.1) node {$c$};
\draw (-4.4,2.1) node {$d$};
\draw (4.4,2.1) node {$a$};

\draw (0,0) node (bd) {$Jbd$};
\draw (0,3.5) node (ab) {$Jab$};
\draw (3.5,0) node (bc) {$Jbc$};
\draw (0,-3.5) node (cd) {$Jcd$};
\draw (-3.5,0) node (ad) {$Jad$};

\draw [quivarrow] (bd)--(ad);
\draw [quivarrow] (bd)--(bc); 
\draw [quivarrow] (cd)--(bd); 
\draw [quivarrow] (ab)--(bd);

\draw [quivarrow] (ad)--(ab);
\draw [quivarrow] (bc)--(ab); 
\draw [quivarrow] (ad)--(cd);
\draw [quivarrow] (bc)--(cd); 
\end{scope}


\draw [doublearrow] (5.25,0) -- (6.75,0);

\end{tikzpicture}
\caption{A geometric exchange (cf.~\cite[Fig.~7]{Sc06})}
\label{fig:geomexch}
\end{figure}

As a consequence, we obtain a canonical isomorphism $C(X_\maxNC,\mat{B}_\maxNC)\isom C(X_{\maxNC'},\mat{B}_{\maxNC'})$ 
by identifying the $X_I$ for $I\in \maxNC\setminus Jac= \maxNC'\setminus Jbd$ and
identifying $X_{Jac}$ and $X_{Jac}^*$ with $X_{Jbd}^*$ and $X_{Jbd}$, respectively.
Thus we have a single cluster algebra in which $X_\maxNC$ and $X_{\maxNC'}$ are two clusters related by mutation.

For any $\maxNC$, there is a homogeneous map 
\[
  \clugr{\maxNC}\colon C(X_\maxNC,\mat{B}_\maxNC) \subseteq  \CC[X_I ^{\pm1}: I\in \maxNC] \to \CC(\Grr)
\]
where $\CC(\Grr)$ is the field of fractions of $\CC[\Grr]$, 
defined by sending each $X_I$ to the corresponding (classical) minor $\minor{I}$, for $I\in\maxNC$.
Since the local quiver for the mutation is as in Figure~\ref{fig:geomexch},
the exchange relation precisely matches the short Pl\"ucker relations between minors
\[
 \minor{Jac}\minor{Jbd} = \minor{Jab}\minor{Jcd} + \minor{Jad}\minor{Jbc}
\]
and so the maps $\clugr{\maxNC}$ and $\clugr{\maxNC'}$ conicide after making the canonical identification just described.
Since every label is in some maximal non-crossing set and all such sets are linked by mutation,
the image of $\clugr{\maxNC}$ contains all the minors and so generates $\CC(\Grr)$ as a field.
But the transcendence degree of $\CC(\Grr)$ is equal to $|\maxNC|$, 
so the minors $\{\minor{J} : J\in\maxNC\}$ must be algebraically independent.
Hence $\clugr{\maxNC}$ is an embedding and identifies $C(X_\maxNC,\mat{B}_\maxNC)$ with 
a single cluster algebra $C(\Grr)\subset \CC(\Grr)$, 
independent of $\maxNC$,
for which all the minors $\minor{J}$ are cluster variables,
so $\CC[\Grr]\subset C(\Grr)$.

It then takes a geometric argument (see \cite[Thm~3]{Sc06}),
applied to a special initial cluster $\{\minor{J} : J\in\maxNC\}$,
to show that this inclusion is an equality and hence $\CC[\Grr]$ is itself a cluster algebra.

\section{The Grassmannian quantum cluster algebra} \label{sec:grQCA}

We can now try to follow the above argument in the quantum case and prove the analogous results as far as possible.
We must start by bringing in the matrix $L_\maxNC$ 
given by the Leclerc-Zelevinsky quasi-commutation rules of Theorem~\ref{thm:LZ}.

\begin{proposition}\label{prop:BSLS}
For every maximal non-crossing set $\maxNC$, the pair  $(\mat{B}_\maxNC, \mat{L}_\maxNC)$ is compatible.
Furthermore, any two such pairs are related by mutation.
\end{proposition}

\begin{proof}
In the category $\CM(A)$, we have a cluster tilting object $T_\maxNC$, as in \eqref{eq:TmaxNC}.
By \cite[Thm~10.3]{BKM}, we have $\mat{B}_\maxNC=\mat{B}(T_\maxNC)$ and, by Theorem~\ref{thm:quasicomm},
we have $\mat{L}_\maxNC=\mat{L}(T_\maxNC)$.
Thus, by the first part of Theorem~\ref{thm:main}, $\mat{B}_\maxNC$ and $\mat{L}_\maxNC$ are compatible.

Given a geometric exchange from $\maxNC$ to $\maxNC'$, the two cluster tilting objects $T_\maxNC$
and $T_{\maxNC'}$ differ by only one indecomposable summand and so must be related
by the mutation of cluster tilting objects(cf.~\cite[Prop 4.5]{GLS06b}).
Thus \cite[Thm 1.4]{OPS} implies that any two such $T_\maxNC$ are related by mutation of cluster tilting objects,
and hence the corresponding pairs $(\mat{B}_\maxNC, \mat{L}_\maxNC)$ are also related by mutation of such pairs,
by the second part of Theorem~\ref{thm:main}.
\end{proof}

As a consequence of this proposition, any maximal non-crossing set $\maxNC$ determines a quantum seed 
$\Phi_\maxNC=(X_\maxNC, \mat{B}_\maxNC, \mat{L}_\maxNC)$ and thus a quantum cluster algebra 
$C_q(\Phi_\maxNC)\subseteq\cT(\mat{L}_\maxNC)$, as in Definition~\ref{def:qca},
which can be graded, for the same combinatorial reasons as the classical case, by giving the initial variables degree 1.
As in the classical case, we have a canonical isomorphism $C_q(\Phi_\maxNC)\isom C_q(\Phi_{\maxNC'})$,
when $\maxNC$ and $\maxNC'$ are related by a geometric exchange.

We now bring back the quantum Grassmannian $\CC_q[\Grr]$, 
as in Definition~\ref{def:qgrass} but with scalars extended from $\qbasering$ to $\basering$.
We denote its skew field of fractions by $\cF(\Grr)$.
The first step in comparing $C_q(\Phi_\maxNC)$ to $\CC_q[\Grr]$ is the following.

\begin{lemma}
\label{lem:two-tori}
The canonical map $\clugr{\maxNC}\colon \cT(\mat{L}_\maxNC)\to \cF(\Grr) \colon X_I\mapsto \qminor{I}$ is injective.
\end{lemma}

\begin{proof}
Note that this map is well-defined, because, by choice, $\mat{L}_\maxNC$ gives the quasi-commutation rules for the $\qminor{I}$ with $I\in\maxNC$.
It suffices to prove that the restriction to the quasi-commuting polynomial subalgebra 
$\basering[X_I : I\in \maxNC]$ is injective.
The image of this restriction lies in $\CC_q[\Grr]$, which is free over $\basering$ with finite rank graded pieces.
Hence the image is free and, since the specialisation to $\vble=1$ is injective, the original map is injective. 
\end{proof}

We claim that the (restricted) maps $\clugr{\maxNC}\colon C_q(\Phi_\maxNC)\to \cF(\Grr)$ are all compatible,
as in the classical case, 
and hence identify each $C_q(\Phi_\maxNC)$ with a single quantum cluster algebra $\grQCA\subseteq\cF(\Grr)$,
which contains clusters of quantum minors $\{\qminor{I} : I\in \maxNC \}$, for every maximal non-crossing set $\maxNC$.
This claim is true because the quantum exchange relation matches precisely with the 
corresponding short quantum Pl\"ucker relation:

\begin{proposition}
\label{prop:newcv=minor} 
Let $a, b, c, d$ be four cyclically ordered indices and $J$ be disjoint from $\{a, b, c, d\}$. 
Let $\clugr{\maxNC}\colon C_q(\Phi_\maxNC)\to \cF(\Grr)$ be the embedding associated to
some maximal non-crossing set $\maxNC$, which contains $Jab, Jbc, Jcd, Jad, Jac$.
Suppose  $\maxNC'$ is obtained by a geometric exchange centred at $Jac$.
Then $\clugr{\maxNC}(X^*_{Jac})=\qminor{Jbd}$.
Thus, given the canonical isomorphism $C_q(\Phi_\maxNC)\isom C_q(\Phi_{\maxNC'})$,
the maps $\clugr{\maxNC'}$ and $\clugr{\maxNC}$ coincide.
\end{proposition}

\begin{proof}
The indices $a,b,c,d$ are cyclically ordered and, 
without loss of generality, we can assume that $a<c$.
Hence the actual ordering of the indices is either $a<b<c<d$ or $d<a<b<c$.
In these two cases, the short quantum Pl\"ucker relation (cf.~\cite{KLR}) can be written, respectively,
\begin{align}
\label{eq:qP1}
\qminor{Jbd} = q^{-1} \qminor{Jac}^{-1}\qminor{Jab}^{~}\qminor{Jcd}^{~}
            + q \qminor{Jac}^{-1}\qminor{Jad}^{~}\qminor{Jbc}^{~},\\
\label{eq:qP2}
\qminor{Jbd} = q^{-1} \qminor{Jad}^{~}\qminor{Jbc}^{~}\qminor{Jac}^{-1}
            + q\qminor{Jcd}^{~}\qminor{Jab}^{~}\qminor{Jac}^{-1}.
\end{align}
We want to show that the right-hand sides are precisely what appears in the 
quantum exchange relation \eqref{eq:qu-exch} for $X^*_{Jac}$,
after identifying $X_I$ with $\qminor{I}$ for $I\in\maxNC$.
In other words, each right-hand side term is a quantum monomial as in \eqref{eq:qu-monom}:
\begin{equation*}
X^{\a}=q^{\gamma(\a)}X_1^{a_1}\dots X_{\cn}^{a_{\cn}},
\end{equation*}
where
\begin{equation*}
 \gamma(\a) = \sfrac12 \sum_{i>j}a_ia_j\lambda_{ij}
\end{equation*}
and $\mat{L}=(\lambda_{ij})$ is the matrix of quasicommutation exponents.
In this case, the entries of this matrix are $c(I_i,I_j)$, as in Definition~\ref{def:non-crossing}.
The precise calculation depends on the order of the indices $a,b,c,d$.
In the case $a<b<c<d$, the first term exponent is
\[
 \gamma= \sfrac12 \bigl( c(Jcd,Jab) - c(Jab,Jac) - c(Jcd,Jac)\bigr) =\sfrac12 (-2-1+1) = -1
\]
and the second term exponent is 
\[
 \gamma=\sfrac12 \bigl( c(Jbc,Jad) - c(Jad,Jac) - c(Jbc,Jac)\bigr) =\sfrac12 (0+1+1) = 1,
\]
as required.
In the case $d<a<b<c$, 
the first term exponent is
\[
 \gamma=\sfrac12 \bigl( c(Jbc,Jad) - c(Jac,Jad) - c(Jac,Jbc)\bigr) = \sfrac12 (-2+1-1) = -1
\]
and the second term exponent is 
\[
 \gamma=\sfrac12 \bigl( c(Jab,Jcd) - c(Jac,Jcd) - c(Jac,Jab)\bigr) = \sfrac12 (0+1+1) = 1.
\]
Each exponent agrees with that in \eqref{eq:qP1} or \eqref{eq:qP2}, as required.
\end{proof}

Because $\grQCA$ contains all quantum minors, we have
\begin{equation}
\label{eq:gr-inc}
\CC_q[\Grr]\subseteq \grQCA
\end{equation}
and consequently the subfield of $\cF(\Grr)$ generated by $\grQCA$ is the whole of $\cF(\Grr)$.

In the quantum case, we do not know any geometric argument to show that the inclusion \eqref{eq:gr-inc}
is an equality. Indeed, all we can prove is the following.

\begin{theorem} 
\label{thm:generic-isom}
The inclusion \eqref{eq:gr-inc} induces an isomorphism 
\[
 \CC_q[\Grr]\otimes_{\basering} \genfld 
 \simeq \grQCA \otimes_{\basering}\genfld.
\]
\end{theorem}

\begin{proof}
We have already noted that $\CC_q[\Grr]$ is a flat deformation of $\CC[\Grr]$ (Theorem~\ref{thm:defgr}).
Furthermore, we know that $\CC[\Grr]$ coincides with the classical cluster algebra $C(X_\maxNC, \mat{B}_\maxNC)$
and with its upper cluster algebra $U(X_\maxNC, \mat{B}_\maxNC)$ (see \cite[Theorem 3]{Sc06} and its proof).
Hence Theorem~\ref{thm:GLSflatdef} and Remark~\ref{rem:GLSflatdef} apply and we see that $\grQCA$ is a flat deformation of 
$C(X_\maxNC, \mat{L}_\maxNC)=\CC[\Grr]$.

Thus the graded pieces of $\CC_q[\Grr]$ and $\grQCA$ have the same rank in equal degrees
and so the inclusion becomes an isomorphism on tensoring with $\genfld$, as required.
\end{proof}

Now, the inclusion \eqref{eq:gr-inc} is an equality in the finite type cases, namely
$\Gr(2,\nn)$, for arbitrary $\nn$, and $\Gr(3,\nn)$, for $\nn=6,7,8$.
This follows, for $\Gr(2,\nn)$, because the cluster variables are all just (quantum) minors and, for these $\Gr(3,\nn)$,
because the additional cluster variables have been computed in \cite[\S3.2]{GLfin} and are explicit polynomials in 
$\CC_q[\Gr(3, \nn)]$.
Note that in the $\Gr(3,\nn)$ cases, the degree $2$ cluster variables do involve (odd) powers of $\vble$,
showing that the extension of scalars to $\basering$ is already necessary here.

On the other hand, for the remaining infinite type cases, 
Theorem~\ref{thm:generic-isom} is the strongest statement we know that could be informally rendered as
``the quantum Grassmannian is a quantum cluster algebra''.
We do not prove the stronger theorem that $\CC_q[\Grr]= \grQCA$
and we do not know of any proof in the literature,
but we have no reason to believe that it couldn't be true.


\end{document}